\documentclass{article}
\usepackage{graphicx} %

\title{Solving Stochastic Fixed-Point Equations with High Probability}
\author{Jelena Diakonikolas\\
Department of Computer Sciences\\
University of Wisconsin-Madison\\
\texttt{jelena@cs.wisc.edu}}
\date{}

\usepackage[margin=1in]{geometry}
 \usepackage{times}
 \usepackage{amsmath,amsfonts,amsthm,bm}

\usepackage{dsfont}
\usepackage{mathtools}
\usepackage[sort]{cite}

\usepackage[colorlinks,urlcolor=blue,citecolor=blue,linkcolor=blue]{hyperref}
 \usepackage{color}

\usepackage{cleveref}

\usepackage{thmtools}
\usepackage{thm-restate}

\usepackage{algorithm}
\usepackage[noend]{algpseudocode}

\usepackage{graphicx}
\usepackage{subcaption}

\newtheorem{theorem}{Theorem}
\newtheorem{lemma}{Lemma}
\newtheorem{corollary}{Corollary}

\newtheorem{proposition}{Proposition}
\newtheorem{assumption}{Assumption}
\newtheorem{fact}{Fact}

\newtheorem{definition}{Definition}

\def\1{\bm{1}}

\def\eps{{\epsilon}}
\def\teps{{\widetilde{\varepsilon}}}

\def\one{{\mathds{1}}}

\def\lambdab{{\overline{\lambda}}}

\def\vzero{{\bm{0}}}

\def\vb{{\bm{b}}}

\def\vx{{\bm{x}}}

\def\vy{{\bm{y}}}

\def\vxi{{\bm{\xi}}}
\def\veta{{\bm{\eta}}}
\def\vetah{\hat{\bm{\eta}}}
\def\etah{\hat{\eta}}
\def\vtau{{\bm{\tau}}}
\def\vsigma{{\bm{\sigma}}}

\def\vxh{\hat{\bm{x}}}

\def\gammab{\overline{\gamma}}

\def\cldel{\mathrm{Cl}^{\gammab}\Delta}
\def\cl{\mathrm{Cl}^{\gammab}}

\def\Rt{\widetilde{R}}

\def\mF{{\bm{F}}}

\def\mS{{\bm{S}}}
\def\mT{{\bm{T}}}
\def\mTt{\widetilde{\bm{T}}}
\def\mTh{\widehat{\bm{T}}}

\newcommand\norm[1]{\left\| #1 \right\|}

\def\cD{{\mathcal{D}}}
\def\cE{{\mathcal{E}}}
\def\cF{{\mathcal{F}}}

\def\cH{{\mathcal{H}}}

\def\cM{{\mathcal{M}}}

\def\cO{{\mathcal{O}}}

\def\sP{{\mathbb{P}}}

\def\sR{{\mathbb{R}}}

\newcommand{\E}{\mathbb{E}}

\DeclareMathOperator*{\argmin}{arg\,min}

\usepackage[colorinlistoftodos]{todonotes}

\newcommand{\prob}{\sP}

\newcommand{\dd}{\mathrm{d}}

\begin{document}

\maketitle

\begin{abstract}
    We study stochastic fixed-point equations $\mT(\vx) = \vx$ over normed spaces $(\cE, \norm{\cdot})$, where the operator $\mT$ is nonexpansive or contractive and is accessed only through unbiased stochastic evaluations with bounded second central moment. Given $\eps > 0, \delta \in (0, 1)$, the goal is to output $\vx \in \cE$ such that  $\norm{\mT(\vx) - \vx} \leq \eps$ with probability at least $1-\delta$. We introduce VR-GHAL, a variance-reduced gradual Halpern method for quadratically smoothable Banach spaces. The key algorithmic ingredient is a recursive stochastic estimator based on clipped differences of oracle evaluations: instead of clipping $\vtau(\vx; \vxi)$ itself, we clip stochastic differences at the Lipschitz scale 
$\gamma\norm{\vx - \vy}$. This makes the estimator pathwise Lipschitz along the algorithmic trajectory while permitting martingale concentration under finite second moments in the native norm. Our main theorem gives an anytime high-probability residual bound: on a single event of probability at least $1 - \delta$, the residual decreases nearly geometrically across epochs, up to lower-order logarithmic factors. Under only bounded variance, displaying only the dependence on the target error $\eps$ and Lipschitz constant $\gamma \in (0, 1]$ of $\mT,$ the resulting oracle complexity is $\min\{\eps^{-5}, (1-\gamma)^{-3}\eps^{-2}\}$. Under a Lipschitz-in-expectation oracle, the dependence improves to the corresponding $\eps^{-3}$
 nonexpansive rate (i.e., for $\gamma = 1$), and under samplewise nonexpansiveness to $\eps^{-2}$. %
\end{abstract}

\section{Introduction}

Fixed-point operator equations $\mT(\vx) = \vx$ are pervasive across scientific disciplines. For instance, they are central to game theory and mathematical programming \cite{facchinei2007finite,yannakakis2009equilibria,chen2025computing,todd2013computation,BauschkeCombettes2017,ParikhBoyd2014}, they arise as primitives in equilibrium problems describing solutions to partial differential equations (PDEs) central to scientific computing \cite{Dudson2025FreeGS,Amorisco2024,liu2021parallel,elman2026surrogate}, and appear in both classical (e.g., reinforcement learning \cite{Bellman1957,Puterman1994,BertsekasTsitsiklis1996,SuttonBarto2018}) and emerging (e.g., deep equilibrium models \cite{bai2019deq} and consistency models \cite{song2023consistency}) areas of machine learning. In many such settings, access to exact operator $\mT$ evaluations is not available due to either computational considerations or the statistical nature of the problem; instead, the operator is accessed through  noisy simulations, sampling, or data. Concrete examples include (i) Bellman equations in reinforcement learning (RL) and stochastic dynamic programming \cite{Bellman1957,Puterman1994,BertsekasTsitsiklis1996,SuttonBarto2018}, (ii) deep equilibrium models and
related implicit-layer architectures in deep learning \cite{bai2019deq, elghaoui2021implicit,
winston2020monotone}, which define their output as the solution to a fixed-point equation solved through a noisy (minibatch) oracle, and (iii) a broad range of applications in extreme-scale scientific computing, such as, for example,  self-consistent field calculations used across materials science and chemistry \cite{ko2023stochastic,kohn1965self}.

Motivated by these applications, we consider the problem of computing an approximate fixed point of an
operator $\mT : \cE \to \cE$ on a real normed vector space $(\cE, \norm{\cdot})$: 
\begin{equation}\label{eq:fp}
  \text{find } \vx \in \cE \text{ such that } \vx = \mT(\vx),
\end{equation}
where $\mT$ is accessed only through an unbiased  \emph{stochastic} oracle $\vtau(\vx;\vxi)$ with bounded variance (more precisely, bounded second central moment), and where $\vxi \sim \cD$ models the underlying randomness from an unknown distribution $\cD$. In particular, we assume
\begin{equation}\label{eq:stochastic-opt-assmpts}
    \E_{\vxi \sim \cD}[\vtau(\vx; \vxi)] = \mT(\vx), \quad \E_{\vxi \sim \cD}[\|\vtau(\vx; \vxi) - \mT(\vx)\|^2] \leq \sigma^2, \quad \forall \vx \in \cE.
\end{equation}
We study such problems in standard settings where the unknown operator $\mT$ is $\gamma$-Lipschitz with $\gamma \in (0, 1]$, meaning that for any $\vx, \vy \in \cE,$ we have $\norm{\mT(\vx) - \mT(\vy)} \leq \gamma \norm{\vx - \vy}$. In particular, the operator is assumed to be either $\gamma$-contractive ($\gamma \in (0,1)$) or nonexpansive ($\gamma = 1$) with respect to
$\norm{\cdot}$. Our goal is to return a point $\vxh$ whose
fixed-point residual $\norm{\mT(\vxh) - \vxh}$\footnote{For contractive operators, $\eps$ fixed-point residual implies $\frac{\eps}{1-\gamma}$ distance to the fixed point. For nonexpansive operators, bounding the distance to fixed points is impossible in general without additional assumptions; see \cite[Theorem 2]{chen2025computing}.} is at most a target $\eps > 0$ \emph{with high probability}, i.e., with probability at least $1-\delta$ for a prescribed confidence $\delta \in (0,1)$, using as few stochastic oracle queries as possible.

There are three requirements for solving such stochastic fixed-point equations that are both natural and, in combination, largely open. First, the aforementioned high-probability guarantee should incur (poly-)logarithmic dependence on the inverse confidence parameter $1/\delta$ in the sample and computational complexities of utilized algorithms to be truly a ``high probability'' guarantee. The reason is that polynomial dependence on $1/\delta$ implied by in-expectation guarantees via Markov inequality leads to either excessively large sample complexity for truly high-probability guarantees, or otherwise forces failure probability $\delta$ to be large, which can hide  heavy upper tails in the fixed-point residual. 

Second, both variance and contraction should be measured in the \emph{native} norm $\norm{\cdot}$ of the space, rather than in $\ell_2$. Operators arising in applications are frequently nonexpansive in a norm other than $\ell_2$---most prominently the $\ell_\infty$ norm in RL  \cite{Bellman1957,Puterman1994,BertsekasTsitsiklis1996,SuttonBarto2018}. Beyond finite-dimensional $\ell_\infty$ problems, fixed-point formulations are a standard language for nonlinear operator equations on function spaces, including differential and integral equations; stochastic fixed-point and splitting methods in Hilbert and Banach spaces likewise model settings where only noisy operator evaluations are available \cite{Zei86,Dei85,CP15,RVV16,mou2022optimal}. 
Observe here that for finite-dimensional $\ell_p$ spaces with $p > 2,$ classical inequalities relating $\ell_p$ norms give $\norm{\cdot}_p \leq \norm{\cdot}_2 \leq d^{\frac{1}{2}-\frac{1}{p}}\norm{\cdot}_p,$ where $d$ is the dimension of the space. Thus, the variance measured in the $\ell_2$ norm can be larger by factors polynomial in the dimension. 
This is precisely the regime one wants to avoid for large state spaces (as in RL applications) and which precludes infinite-dimensional spaces altogether. 

Finally, 
 we ask for a \emph{parameter-free, anytime}
algorithm: beyond the confidence parameter $\delta$, total number of iterations/epochs $K,$ and the Lipschitz constant $\gamma$, the method should require no knowledge of problem parameters such as the variance proxy or the distance to a solution, and the guarantee on the residual should hold for whatever number of epochs the algorithm is run, without fixing the target error $\eps$ upfront. 

\subsection{Technical Overview \& Contributions}\label{sec:contributions}

This paper introduces a new variance reduction strategy that satisfies all three requirements listed above, while leading to the (sample/stochastic oracle) complexity guarantees on par with the lowest available in the literature \cite{bravo2024stochastic} (which apply in expectation, under bounded $\ell_2$ variance, explicitly assume that the iterates remain in a bounded set, and are not parameter-free as defined above). The algorithmic methodology we employ builds upon the Gradual Halpern Algorithm (GHAL), recently introduced in \cite{diakonikolas2025pushing}. GHAL uses, as a subroutine, an iteration defined by:
\begin{equation}\label{eq:main-iteration}\tag{FixHal}
    \vx_{k+1} = \lambda \vx_0 + (1-\lambda) \mT(\vx_k).%
\end{equation}
There are two ways to interpret this iteration. First is as a special case of classical Halpern iteration $\vx_{k+1} = \lambda_k \vx_0 + (1-\lambda_k) \mTt(\vx_k)$ \cite{halpern1967fixed}, where instead of a decreasing schedule, the step size $\lambda_k \in (0, 1)$ is forced to be constant. The second is as solving a ``regularized'' problem, applying Banach-Picard iteration to the operator $\mF(\vx) = \lambda \vx_0 + (1-\lambda)\mT(\vx).$ Observe that if $\mT$ is $\gamma$-Lipschitz, then $\mF$ is $(1-\lambda)\gamma$-Lipschitz (more contractive). 

GHAL alternates between decreasing $\lambda$ by a fixed constant factor and running \eqref{eq:main-iteration} until the fixed-point residual decreases by a constant factor. In that sense, it can be viewed as a Halpern-style iteration with a gradual step size decrease, or as an inexact proximal point (or resolvent)-style method with a geometrically decreasing regularization schedule. Regarding the latter interpretation, it is worth noting that in the present context of fixed-point equations and unlike in minimization contexts where proximal point-based strategies are traditionally used, this methodology applies to arbitrary normed spaces and even generalizes to metric spaces with a convex metric; see \cite{diakonikolas2025pushing}. 

Our approach is based on using GHAL with a novel variance-reduced estimator for $\mT$ introduced in this work. We note here that simply bounding the error of the estimator and plugging it into the existing analysis of GHAL is insufficient here. The reason is that GHAL (or its adaptive version AdaGHAL)  relies on the algorithm iterates remaining in a bounded set (with AdaGHAL adaptively estimating its diameter), which is guaranteed for nonexpansive operators under \emph{exact} operator evaluations as none of the updates can increase the iterate distance to the operator fixed points. However, under \emph{inexact/stochastic} operator evaluations, it is unclear how to guarantee that the iterates remain in a bounded set without imposing additional assumptions on the problem. This issue is particularly pronounced when the operator is nonexpansive (or when its contraction factor is close to one), due to the inability to rely on contraction to correct for stochastic errors. The reason is that the stochastic operator evaluations are not necessarily nonexpansive or even Lipschitz continuous. Controlling stability and boundedness of stochastic iterates is a classical difficulty in stochastic approximation for nonexpansive maps and Q-learning \cite{abounadi2002stochastic}, and it also appears explicitly in recent stochastic fixed-point analyses \cite{bravo2024fixedpoint,bravo2024stochastic}.

The first technical contribution is showing that convergence of GHAL can be guaranteed without knowledge or explicit estimation of the iterate distance to fixed points (or any bound on the diameter of the set containing the algorithm iterates), while remaining robust to (sufficiently slow) iterate drifts away from the fixed points. This is achieved through an intricate inductive argument ensuring that the iterate drift is offset by the progress in reducing the fixed-point residual. Thus, even if the stochastic operator were to be replaced by an exact one, the provided analysis would still be new, as it does not require either knowledge (as in GHAL) or estimation (as in AdaGHAL) of the diameter of the set containing the iterates.%

A core contribution of this paper is the design of a new variance-reduced operator (\Cref{sec:VR+master-thm}). The broad structure of the estimator is built on recursive variance reduction of SARAH/SPIDER type, originally developed for stochastic gradient estimation to solve minimization problems \cite{nguyen2017sarah,fang2018spider}. This estimator is compatible with the GHAL architecture, as both GHAL and SARAH/SPIDER estimators operate in epochs. In particular, high accuracy estimation can naturally coincide with the GHAL epoch starts, while the estimation based on the differences of minibatch stochastic estimators can be combined with the analysis of \eqref{eq:main-iteration} over the epoch, which tracks successive iterate distances. The key idea in our work is the use of a \emph{clipped difference} estimator. Rather than clipping stochastic operator values $\vtau(\vx;\vxi)$ directly (as is common in the literature employing gradient clipping in optimization \cite{gorbunov2020clipping,nguyen2023clippedsgd,cutkosky2021highprob,zhang2022parameter,sadiev2023highprob}), we estimate differences of operator evaluations between successive iterates $\vy_{j-1}, \vy_j$; namely
\(
    \mT(\vy_j)-\mT(\vy_{j-1}), 
\) 
through stochastic differences
\(
    \Delta(\vy_j,\vy_{j-1})
    =
    \vtau(\vy_j;\vxi_1)-\vtau(\vy_{j-1};\vxi_2),
\)
and clip these differences at a radius dictated by the Lipschitz bound on $\mT$:
\(
    \gamma \|\vy_j-\vy_{j-1}\|.
\) 
This choice is crucial. It makes the stochastic estimator behave, pathwise, like a Lipschitz operator along the algorithm trajectory, which is exactly the property needed by the GHAL analysis. At the same time, it allows martingale-based concentration arguments under only finite second moments in the native norm. We note here that while clipping of a difference operator \emph{on a fixed clipping schedule} appears in the context of distributed stochastic convex optimization and monotone variational inequalities \cite{gorbunov2024composite}, we are not aware of prior work explicitly leveraging Lipschitz properties to perform the clipping and carry out the stochastic error analysis.  

Our main result, stated in \Cref{thm:master}, is a master theorem for high-probability convergence of the resulting variance-reduced GHAL (VR-GHAL, \Cref{algo:GHAL}), in \emph{quadratically smoothable} spaces (see \Cref{asspt:Banach}; these spaces include separable smooth Banach spaces of possibly infinite dimension and all finite-dimensional spaces using a quantifiable notion of regularity introduced in \cite{juditsky2008large}). The theorem provides sufficient conditions on the variance of clipped difference operators that lead to near-geometric reduction in the fixed-point residual over the algorithm epochs, which further translates into the number of epochs to reach fixed-point residual of order $\epsilon > 0$ being the sum of a logarithmic in $1/\epsilon$ term (with leading constant one) and lower order (log-log and constant) terms. The theorem is proved using an intricate inductive argument coupling the control of the stochastic error (in high probability) with a near-geometrically decreasing bound on the residual. The established guarantee is \emph{anytime}---the algorithm succeeds with probability $1-\delta$ over any number of epochs with a near-geometric decrease in the residual, without the need to pre-determine the target residual $\eps.$ Further, the only problem parameters needed in the algorithm are the number of epochs for which it is run (or, alternatively, an exit criterion based on target error), the failure probability $\delta,$ and the Lipschitz constant $\gamma \in (0, 1].$ 

To obtain sample/oracle complexity results, it then suffices to apply \Cref{thm:master}, showing that the sufficient condition on the difference operator variance can be satisfied using simple minibatching. In particular, without any additional assumptions about the stochastic oracle beyond what is stated in \eqref{eq:stochastic-opt-assmpts}, the sample complexity of VR-GHAL is $\widetilde \cO\big(\tfrac{D^2}{\eps^2} + \min\{\tfrac{D^6}{\eps^5}, \tfrac{D^3}{(1-\gamma)^3\eps^2}\}\big)$  (\Cref{cor:sample-complexity-gen}), where
  $D = \norm{\vx_* - \vx_0} + \sigma$. Assuming, in addition, that the operator is Lipschitz in expectation (see \Cref{asspt:Lipschitz-in-expectation}), the scaling with $\eps$ improves to
  order-$\eps^{-3}$ for nonexpansive operators (\Cref{cor:exp-Lip}). This scaling improves to $\eps^{-2}$ for
  operators that are nonexpansive per sample
  (\Cref{sec:other-models}). All bounds apply to quadratically smoothable spaces, as noted above. In particular, for standard $\ell_p$ spaces with $p \geq 2,$ the bounds scale polynomially with $\min\{p-1, \ln(d)\},$ where $d$ is the ambient dimension.

\subsection{Related Work}

Our work is related to literature in several distinct areas, including fixed-point solvers, robust mean estimation, concentration in Banach spaces, gradient clipping in stochastic optimization, and recent work on solving stochastic Bellman equations. Below, we review those works that are most closely related to ours.  

\paragraph{Halpern iteration and related methods.}
Classical Halpern iteration was introduced in \cite{halpern1967fixed}. Its convergence properties for solving deterministic fixed-point equations with access to exact operator evaluations have been established in recent literature, over two lines of work: considering general Banach spaces \cite{sabach2017first,
kohlenbach2011quantitative, leustean2007rates, cheval2023modified,contreras2023optimal} and considering the special case of Euclidean/Hilbert spaces \cite{lieder2021convergence,tran2021halpern,lee2021fast,diakonikolas2020halpern,kim2019accelerated,he2024convergence,yoon2021accelerated,diakonikolas2021potential}, where connection to monotone variational inequalities can be exploited. On the algorithm side, our work builds on the recently introduced gradual Halpern framework \cite{diakonikolas2025pushing}, as discussed in \Cref{sec:contributions}.  

\paragraph{Stochastic fixed-point problems and algorithms.}
Stochastic fixed-point equations arise naturally when the operator defining the problem is an expectation over simulation randomness. This viewpoint includes Bellman equations in dynamic programming and RL  \cite{Bellman1957,Puterman1994,BertsekasTsitsiklis1996,SuttonBarto2018}. Classical stochastic-approximation analyses for nonexpansive maps include \cite{abounadi2002stochastic}, which proves almost-sure convergence via ODE/stability arguments and applies the framework to Q-learning. These results are asymptotic and rely on stability/boundedness arguments, whereas our goal is a finite-sample high-probability residual guarantee. Much of the recent work on solving stochastic fixed-point equations with formal sample complexity guarantees has been tailored to RL applications, on one end addressing $\ell_\infty$ settings (going beyond the usual Euclidean-based arguments), but on the other crucially exploiting strong additional structural properties of such problems such as samplewise nonexpansiveness  \cite{lee2025optimal,lee2025near,wainwright2019stochastic}. The goal of our work is not to develop new RL algorithms, but to address general stochastic fixed-point equations with high probability. We note however that under samplewise nonexpansiveness, our sample complexity scales with $\eps^{-2},$ as in the aforementioned line of work.

Several works give finite-sample guarantees for stochastic approximation for \emph{contractive} operators ($\gamma < 1$) in non-Euclidean geometries. For example, \cite{chen2020finite} analyze stochastic approximation with contractive operators in arbitrary norms using a generalized Moreau-envelope Lyapunov function, with applications to RL and with logarithmic dimension dependence in $\ell_\infty$ settings. \cite{mou2022optimal} develop a broad variance-reduced stochastic approximation theory for contractive fixed-point equations in separable Banach spaces, with instance-dependent nonasymptotic guarantees in arbitrary seminorms and applications to several RL problems. Their framework is closest to ours in its use of Banach-space geometry, but it is centered on contractive operators (or multi-step contractivity for linear operators) and assumes sample-level Lipschitzness together with bounded noise at the fixed point. Our focus is different: we allow merely nonexpansive operators, require only bounded second central moment in the native norm in the base model, and obtain high-probability residual bounds through a clipped-difference estimator.

A separate line of work has addressed stochastic monotone inclusion problems in Euclidean spaces, which (for Euclidean/Hilbert spaces only) are closely related (equivalent up to translation and rescaling) to fixed-point equations with monotone operators. In particular, under a condition comparable to \eqref{eq:stochastic-opt-assmpts} and additional Lipschitz-in-expectation assumption (see \Cref{asspt:Lipschitz-in-expectation}), \cite{cai2022stochastic} obtained sample complexity scaling with $\eps^{-3}$ for a guarantee that holds in expectation. We obtain the same sample complexity scaling under effectively equivalent assumptions, but in much broader generality: applying with high probability and in general quadratically smoothable spaces. We also note that for Euclidean spaces, lower sample complexity of the order $\eps^{-2}$ was obtained in \cite{chen2022near} without requiring Lipschitzness in expectation\footnote{\cite{chen2022near} writes their results in the context of monotone variational inequalities, with a guarantee on the operator norm, which is precisely the setting of monotone inclusion. Since every cocoercive operator is monotone and Lipschitz operator, and since minimizing the norm of a cocoercive operator is equivalent to minimizing the residual of a nonexpansive operator (see, e.g., the discussion in \cite{diakonikolas2020halpern}), their result also implies the stated sample complexity for stochastic fixed-point equations with nonexpansive operators.}. It is unclear whether such scaling can be transferred to other norms; in fact, there is evidence that this may not be possible \cite{bravo2024stochastic}. 

The works most closely related to ours are  \cite{bravo2024fixedpoint,bravo2024stochastic}. Both works allow the operator to be nonexpansive in a general finite-dimensional norm, but the stochastic noise is controlled through Euclidean second moments. Their explicit high-probability consequences are Markov-based and therefore have polynomial, rather than (poly-)logarithmic, dependence on the failure probability. In terms of the sample complexity results, and suppressing the dependence on other problem parameters, \cite{bravo2024fixedpoint} leads to sample complexity of the order $\eps^{-6}$ for nonexpansive operators. The subsequent work \cite{bravo2024stochastic} improves this scaling to  $\eps^{-5}$ for nonexpansive operators and $(1-\gamma)^{-3}\eps^{-2}$ for $\gamma$-contractive operators, which is of the same order as the bounds we obtain with high probability and with variance measured in the native norm. We further note that unlike \cite{bravo2024stochastic}, we do not assume that the iterates remain in a bounded set. 

\paragraph{High-probability guarantees and clipping under heavy tails.} 
Gradient clipping is the standard mechanism for converting bounded- or
heavy-tailed-variance assumptions into high-probability guarantees with
logarithmic confidence dependence, rooted in the statistical literature on trimmed/truncated mean estimation \cite{catoni2018dimension,lugosi2021robust,oliveira2019sub,tukey1963less}.  It has been employed in a series of works on
stochastic optimization and variational inequalities to obtain such high-probability guarantees
\cite{gorbunov2020clipping, nguyen2023clippedsgd, sadiev2023highprob,
gorbunov2022clipped, gorbunov2024composite, zhang2022parameter,cutkosky2021highprob}. Our
estimator is most directly related to the \emph{difference}-clipping
 in \cite{gorbunov2024composite}, where stochastic gradient
\emph{differences} are clipped to preserve convergence in composite and distributed
problems. We depart from this literature in two ways. First, all of
these analyses are carried out in standard Euclidean space; we work in the native
norm of a quadratically smoothable space, which requires
Banach-space martingale concentration
\cite{pinelis1994optimum, juditsky2008large, whitehouse2024mean} in
place of the usual $\ell_2$ arguments. Second, and more fundamentally,
prior difference-clipping sets the clipping threshold from the noise
scale, whereas we set it from the \emph{operator's Lipschitz constant}:
because $\E[\Delta(\vx,\vy)] = \mT(\vx)-\mT(\vy)$ has norm at most
$\gamma\norm{\vx-\vy}$, clipping $\Delta(\vx,\vy)$ at radius $\gamma\norm{\vx-\vy}$
controls the bias and variance while exploiting structure that the
optimization-side clipping methods do not have access to.

High-probability and concentration guarantees have also been developed for contractive stochastic approximation. For instance, \cite{chandak2022concentration} derived martingale-based concentration bounds for stochastic approximation with contractive maps and martingale-difference or Markov noise, with applications to asynchronous Q-learning. More recently, \cite{chen2025concentration} proved maximal concentration inequalities for contractive stochastic approximation in arbitrary norms under sub-Gaussian additive noise or bounded multiplicative noise. These results are closest in spirit to our high-probability objective, but they rely on contractivity and stronger tail/noise structure, while our bounds allow for merely finite second moments in the native norm and cover nonexpansive operators.

\paragraph{Mean estimation and concentration in Banach spaces.}
Our high-probability analysis rests on concentration for vector-valued
martingales in smooth Banach spaces. We use the supermartingale
construction of \cite{pinelis1994optimum} and its recent adaptation
for infinite-variance and martingale-dependent data by
\cite{whitehouse2024mean}, together with the $2$-smooth /
$\kappa$-regular space machinery of \cite{juditsky2008large,
juditsky2025aggregating}, which motivated the quadratically smoothable abstraction. At each epoch start, our algorithm uses the classical median-of-means estimator
\cite{nemirovsky1983problem, alon1996space, hsu2016loss}, while we note
that geometric-median-type estimators \cite{minsker2015geometric,
whitehouse2024mean} with better universal constants can be substituted
without changing the rates.

\section{Preliminaries}\label{sec:prelims}

Throughout the paper, we use $(\cE, \norm{\cdot})$ to denote a real normed vector space. We use boldface notation to denote the elements of $\cE,$ and, in particular, $\vzero$ to denote the zero element of this space. We define $[m] := \{1, \dots, m\}.$ 

\subsection{Problem Setup}

As stated in the introduction, the central goal is to solve a fixed point operator equation \eqref{eq:fp}, where $\mT$ is either a contractive or a nonexpansive operator. Whenever $\mT$ is nonexpansive,\footnote{For contractive operators, existence of fixed points is guaranteed by the classical Banach theorem \cite{banach1922operations}.} we assume that the set of solutions to \eqref{eq:fp} is nonempty and denote a generic fixed point by $\vx_*$. None of the results are specific to the choice of a fixed point; thus all will apply to any chosen fixed point. 

\paragraph{Normed vector space.} 
We consider stochastic fixed-point equations defined on a real normed vector space $(\cE, \norm{\cdot})$ that is quadratically smoothable, as defined below. %

\begin{definition}[Quadratically Smoothable Space]\label{asspt:Banach}
    We say that $(\cE, \norm{\cdot})$ is quadratically smoothable if it is a real separable Banach space and the following conditions both hold:
    \begin{itemize}
        \item There exists a compatible norm $\norm{\cdot}_+$ and a constant $C \geq 1$ such that
    \begin{equation}\label{eq:norm-equivalence}
        \norm{\vx} \leq \norm{\vx}_+ \leq C\norm{\vx}, \quad \forall \vx \in \cE.
    \end{equation}
    \item The map $\vx \mapsto \norm{\vx}_+^2$ is continuously Fr\'{e}chet differentiable and satisfies for some fixed $B\geq 1$ and all $\vx, \vy \in \cE:$
    \begin{equation}\label{eq:norm-smoothness}
        \norm{\vx + \vy}_+^2 \leq \norm{\vx}_+^2 + D(\norm{\cdot}_+^2)(\vx)[\vy] + B\norm{\vy}_+^2,
    \end{equation}
    where $D(\norm{\cdot}_+^2)(\vx)[\vy]$ denotes the Fr\'{e}chet derivative of $\norm{\cdot}_+^2$ evaluated at $\vx,$ in the direction of $\vy.$ 
    \end{itemize}
    We refer to $\kappa_\cE := B C^2$ as the smoothness parameter of $(\cE, \norm{\cdot})$, provided this space is quadratically smoothable.
\end{definition}
Throughout the paper, whenever we make the assumption that $(\cE, \norm{\cdot})$ is quadratically smoothable, we fix a witness $(\norm{\cdot}_+, B, C)$ satisfying \eqref{eq:norm-equivalence}, \eqref{eq:norm-smoothness}, and write $\kappa_\cE := BC^2$ for the associated smoothness parameter. 

\Cref{asspt:Banach} covers a broad range of real vector spaces. In particular, it encompasses all $(2, B)$-smooth separable Banach spaces in the special case of $\norm{\cdot} = \norm{\cdot}_+$ and $C = 1$ (in this case, \Cref{asspt:Banach} reduces to the definition of $(2, B)$-smooth separable Banach spaces). Additionally, \Cref{asspt:Banach} includes all (finite-dimensional) $\kappa$-regular spaces as defined by \cite{juditsky2008large}. As a result, all of the following example spaces are covered by our results:
\begin{itemize}
    \item All finite-dimensional normed vector spaces of dimension $d$; in this case $\kappa_\cE \leq d$ \cite[Example 3.1]{juditsky2008large};
    \item All finite-dimensional $\ell_p$ spaces with $p \geq 2$ and $d \geq 3;$ here, $\kappa_\cE \leq  \min\{p-1, 2\ln(d)\}$ \cite[Example 3.2]{juditsky2008large};
    \item All (finite-dimensional) Schatten-$p$ spaces (spaces of real $k \times d$ matrices with Schatten-$p$ norm defined as the $\ell_p$ norm of the matrix singular values) with $p \geq 2$; in this case $\kappa_\cE \leq \min\{\max\{2, p-1\}, (2\ln(\min\{k, d\}+2)-1)\exp(1)\}$ \cite[Example 3.3]{juditsky2008large};
    \item All infinite-dimensional $\ell_p$ spaces and separable $L_p(\mu)$ spaces with $p \in[2, \infty)$; in this case one can take $B = p-1,$ $C=1$ and thus $\kappa_\cE = p-1$ \cite{pinelis1994optimum}.
\end{itemize}
We also refer to \cite{juditsky2025aggregating} for additional examples in finite-dimensional spaces and their generalizations. 

Throughout the paper, whenever we work with any random variables in $\cE,$ we assume that they are Bochner integrable, and that all  (conditional) expectations are Bochner expectations. This is a standard technical assumption employed for Banach spaces; see, e.g., \cite{pinelis1994optimum,whitehouse2024mean}. 

\paragraph{Stochastic oracle.} 
As mentioned earlier, we consider stochastic operator evaluations, where we have access to a stochastic oracle $\vtau(\vx; \vxi)$ with $\vxi$ drawn i.i.d.\ from an unknown distribution $\cD$ and independently of the algorithmic past,  where $(\vx, \vxi)\mapsto \tau(\vx; \vxi)$ is jointly measurable and such that for some constant $\sigma^2 \geq 0,$ \eqref{eq:stochastic-opt-assmpts} holds. 

We emphasize here, that unlike prior work \cite{bravo2024stochastic}, we \emph{do not} require that the variance is bounded w.r.t.\ the $\ell_2$ norm; instead, we can use the norm associated with the considered space. This is essential for avoiding (polynomially) dimension-dependent factors resulting from the conversion between the $\ell_2$ norm and the vector space-endowed norm $\norm{\cdot}$ and for handling even potentially infinite-dimensional vector spaces. %

We provide results considering different assumptions on $\vtau$, from no additional assumption compared to \eqref{eq:stochastic-opt-assmpts}, over mild assumptions such as the ability to query two points $\vx, \vy$ with the same random seed $\vxi$ and requiring ``Lipschitzness in expectation'' for the operator estimate $\vtau,$ to assumptions imposing nonexpansivity or even contractivity of  $\vtau(\cdot; \vxi)$ for any fixed random seed $\vxi,$ motivated by applications in reinforcement learning. These optional and often mild additional assumptions are summarized below. 

Whenever we make additional assumptions about the regularity (Lipschitzness)  of the stochastic oracle, we assume that we have access to a multi-query oracle, as summarized below.

\begin{assumption}[Multi-query Stochastic Oracle]\label{asspt:multi-query-oracle}
    It is possible to query the stochastic oracle $\vtau$ at two or more $\vx, \vy \in \cE$ using the same random seed $\vxi$. 
\end{assumption}

The two optional assumptions about regularity of the stochastic oracle are summarized in the following.

\begin{assumption}[Lipschitzness in Expectation]\label{asspt:Lipschitz-in-expectation}
    There exists $L \in (0, + \infty)$ such that $\E_{\vxi \sim \cD}[\|\vtau(\vx; \vxi) - \vtau(\vy; \vxi)\|^2] \leq L^2 \|\vx - \vy\|^2,$ $\forall \vx, \vy \in \cE.$
\end{assumption}

\begin{assumption}[Nonexpansive/Contractive Stochastic Oracle]\label{asspt:samplewise-Lip}
    For $\cD$-almost every $\vxi$ and any $\vx, \vy \in \cE$, $\|\vtau(\vx; \vxi) - \vtau(\vy; \vxi)\| \leq \gamma \|\vx - \vy\|,$ where $\gamma \in (0, 1].$
\end{assumption}

\subsection{Estimating the Norm of Sample Averages and Sums}

We first state the following simple lemma, which establishes that in the considered quadratically smoothable vector spaces, minibatching reduces the variance. This is essential for establishing concentration results needed for our analysis. The proof of the lemma is provided in \Cref{appx:omitted-prelims}, for completeness.   
\begin{restatable}{lemma}{lemmaminibatchvar}\label{lemma:minibatching-variance}
    Let $(\cE, \norm{\cdot})$ be a quadratically smoothable  space (\Cref{asspt:Banach}). Let $\cF_0 \subseteq \cF_1 \subseteq \dots \subseteq \cF_m$ be a filtration, and let $\veta_i \in L^2(\Omega; \cE)$ be $\cF_i$-measurable random variables that satisfy $\E[\veta_i|\cF_{i-1}] = \vzero,$ $i \in 1, \dots, m.$ Let $\vetah_m = \frac{1}{m} \sum_{i=1}^m \veta_i.$ Then:
    \begin{equation}\notag%
        \E[\|\vetah_m\|^2|\cF_0] \leq \frac{\kappa_\cE}{m}\Big(\frac{1}{m}\sum_{i=1}^m \E[\|\veta_i\|^2|\cF_0]\Big). 
    \end{equation}
    More generally, if there are nonnegative, integrable, $\cF_{i-1}$-measurable $\sigma_i^2$, $i \in [m]$, such that $\E[\|\veta_i\|^2|\cF_{i-1}] \leq \sigma_i^2$ almost surely; then
    \begin{equation}\notag%
        \E[\|\vetah_m\|^2|\cF_0] \leq \frac{\kappa_\cE}{m}\Big(\frac{1}{m}\sum_{i=1}^m \E[\sigma_i^2|\cF_0]\Big). 
    \end{equation}
\end{restatable}

To obtain our high-probability results, we rely upon concentration results for martingale sequences in smooth Banach spaces; namely on \cite[Theorem 3.2]{pinelis1994optimum} and its recent adaptation in \cite[Proposition 3.3]{whitehouse2024mean}. Since neither reference provides a concentration bound in a form that is used in our work, we state it in \Cref{prop:Banach-concentration} below, and prove it, for completeness, in \Cref{appx:omitted-prelims}. This specific form allows us to obtain a time-uniform bound (based on an application of classical Ville's inequality), which in turn is useful for avoiding applications of a union bound to bound the joint probability of a sequence of martingales exceeding a pre-specified quantity. %

\begin{restatable}{proposition}{propbanachconc}\label{prop:Banach-concentration}
    Let $(\cE, \norm{\cdot})$ be a quadratically smoothable space (\Cref{asspt:Banach}), and fix the associated constants $C, \kappa_\cE$. Let $n \geq 1$ and let $\veta_i \in \cE$, $i \in [n]$, define a stochastic process adapted to filtration $\{\cF_i\}_{0\leq i\leq n}$, such that $\E[\veta_i|\cF_{i-1}] = \vzero$ and for $\cF_0$-measurable $\sigma_i^2$ and $b_i$, it holds almost surely that $\E[\|\veta_i\|^2 | \cF_{i-1}] \leq \sigma_i^2 < \infty$ and $\norm{\veta_i} \leq b_i < \infty$. Define $\mS_0 = 0$, $\mS_m = {\sum_{i=1}^m \veta_i}$ for $m \geq 1,$ and $\norm{\vb_{1:n}}_{\infty} := \max_{1 \leq i \leq n}b_i.$ Then, for any $\delta \in (0, 1)$ and any $n \geq 1$, we have that with probability at least $1-\delta$ for all $m \in [n]$ simultaneously:
    \begin{equation}\notag
        \norm{\mS_m}\leq \sqrt{2\kappa_\cE \ln(2/\delta) \sum_{i=1}^n \sigma_i^2 } + \frac{C\ln(2/\delta) \norm{\vb_{1:n}}_{\infty}}{3}.
    \end{equation} 
\end{restatable}

\subsection{Median-of-Means Estimator}\label{sec:mom}

The median-of-means estimator (MoME) is a classical estimator originally introduced in \cite{nemirovsky1983problem} (although not known under this name until it was rediscovered later; see \cite{alon1996space,hsu2016loss}). MoME can be seen as a form of ``boosting'' a weak estimator that succeeds with constant probability to one that succeeds with probability $1-\delta,$ for $\delta \in (0, 1),$ at logarithmic cost in $1/\delta.$ 
We summarize here known results about the MoME, using a strengthening of the results to the ``unknown'' error margin and slightly tighter constants presented in \cite{hsu2016loss}, which is useful for developing a  parameter-free algorithm. More detailed statements and proofs can be found in \cite{hsu2016loss}. 

\paragraph{One-dimensional MoME} 

The MoME utilizes i.i.d.\ samples split into $k = \Theta(\log(1/\delta))$ `buckets,' each containing $m$ samples (so the total number of samples is $mk$). The estimator is summarized in \Cref{algo:MoME-1D}. Notation $\etah_{\mathrm{med}}:= \mathrm{median}(\etah_1, \etah_2, \dots, \etah_k)$ refers to any median (if not unique). 

\begin{algorithm}
\caption{Median-of-Means Estimator, 1D version (MoME-1D)}\label{algo:MoME-1D}
    \begin{algorithmic}
        \State \textbf{Input}: $m \geq 4$, $\delta \in (0, 1),$ sample access to $\eta_i \sim \cD$ drawn i.i.d., where $\E[\eta_i] = \eta$, $\E[(\eta_i - \eta)^2] = \sigma^2$
        \State $k = \lceil 4.5 \ln(1/\delta)\rceil$
        \For{$j = 1$ to $k$}
        \State Draw $m$ i.i.d.\ samples from $\cD$ and let $\etah_j$ be their empirical average
        \EndFor
        \State\Return $\etah_{\mathrm{med}}:= \mathrm{median}(\etah_1, \etah_2, \dots, \etah_k)$
    \end{algorithmic}
\end{algorithm}

\begin{lemma}\label{lemma:1D-MoME}
    Given $m \geq 4$, $\delta \in (0, 1)$, and sample access to i.i.d.\ examples from some distribution $\cD,$ consider the output $\etah_{\mathrm{med}}:= \mathrm{median}(\etah_1, \etah_2, \dots, \etah_k)$ of \Cref{algo:MoME-1D}. Then, with probability at least $1-\delta,$
    \[|\etah_{\mathrm{med}} - \eta| \leq 6\sigma\sqrt{\frac{\lceil \ln(1/\delta) \rceil}{mk}}.\]
\end{lemma}

\paragraph{MoME in Banach Spaces}

The variant of MoME presented in \cite{hsu2016loss} does not explicitly use Chebyshev's inequality, since the property that the variance of a sample average of $m$ i.i.d.\ samples scales with $1/m$ only holds for Rademacher type-2 spaces. In particular, this is true for quadratically smoothable spaces (as in \Cref{lemma:minibatching-variance}), but fails in infinite-dimensional $\ell_p$ spaces with $p < 2.$ The algorithmic procedure from \cite{hsu2016loss} is summarized in \Cref{algo:MoM}. We state the result from \cite{hsu2016loss} assuming access to weak estimators $\vetah_1, \vetah_2, \dots, \vetah_k$, which estimate the mean to error $\epsilon > 0$ with constant probabilities.  We then apply \Cref{lemma:minibatching-variance} to argue such estimators can be obtained using sample averages.

\begin{algorithm}
\caption{Median-of-Means Estimator (MoME)}\label{algo:MoM}
    \begin{algorithmic}
        \State \textbf{Input}: $\delta \in (0, 1),$ independent estimators $\vetah_1, \dots, \vetah_k$ such that $\forall i \in [k],$  $\prob[\norm{\vetah_i - \veta} \geq \epsilon] \leq 1/3$
        \State Let  $\cH := \{\vetah_1, \vetah_2, \dots, \vetah_k\}$
        \State For each $i \in [k],$ let $r_i := \inf\{r \geq 0: \#\{j \in [k]: \norm{\vetah_j - \vetah_i} \leq r\} > k/2\}$
        \State Let $i_* \in \argmin_{i \in [k]}r_i,$ choosing the lowest such $i_*$ if not unique
        \State \Return $\vetah := \vetah_{i_*}$
    \end{algorithmic}
\end{algorithm}

The boosting result from \cite{hsu2016loss} is summarized in the following lemma.
\begin{lemma}[Proposition 9 in \cite{hsu2016loss}]\label{lemma:MoME-HsuSabato}
   If $k \geq  18 \ln(1/\delta),$ then with probability at least $1-\delta,$ the output of \Cref{algo:MoM} satisfies $\norm{\vetah - \veta} \leq 3\epsilon.$ 
\end{lemma}

To apply \Cref{lemma:MoME-HsuSabato}, we use it in conjunction with \Cref{lemma:minibatching-variance}, which provides weak estimators based on simple sample averages. The final guarantee used in our results is stated below.

\begin{proposition}\label{prop:MoME-final}
    Let $(\cE, \norm{\cdot})$ be a quadratically smoothable space (\Cref{asspt:Banach}). Given $m = n k$ i.i.d.\ samples $\veta_1, \veta_2, \dots, \veta_m$, let $\vetah_i := \frac{1}{n}\sum_{j = (i-1)n + 1}^{i n} \veta_j,$ for $i \in [k],$ suppose that $\E[\veta_j] = \veta$ and $\E[\norm{\veta_j - \veta}^2] \leq \sigma^2 < \infty$ for $j \in [m]$. If $k \geq 18 \ln(1/\delta),$ then with probability at least $1-\delta,$ the output $\vetah$ of MoME (\Cref{algo:MoM}) applied to $\vetah_1, \dots, \vetah_k$ satisfies $\norm{\vetah - \veta} \leq 3\sigma\sqrt{\frac{3\kappa_\cE}{n}}.$ As a result, given $\epsilon > 0$ and $\delta \in (0, 1)$, and choosing $n = \lceil \frac{27}{\epsilon^2} \rceil$ and $k = \lceil 18\ln(1/\delta)\rceil$, we have that with probability at least $1 - \delta,$ $\norm{\vetah - \veta} \leq \eps \sqrt{\kappa_\cE}\sigma$. The total number of samples utilized in this case is $m = nk %
    = \cO(\frac{\ln(1/\delta)}{\eps^2})$.  
\end{proposition}
\begin{proof}
    For quadratically smoothable spaces and under the stated variance assumption, \Cref{lemma:minibatching-variance} applies to centered random variables $\veta_i - \veta$. Thus, for each $\vetah_i,$ we have $\E[\norm{\vetah_i - \veta}^2] \leq \frac{\kappa_\cE \sigma^2}{n}.$ By Chebyshev inequality, $\prob[\norm{\vetah_i - \veta} \geq \sqrt{3\kappa_\cE\sigma^2 /n}] \leq \frac{1}{3}.$ It remains to apply \Cref{lemma:MoME-HsuSabato} with $\epsilon = \sqrt{3\kappa_\cE\sigma^2 /n}.$ The remaining claims follow as a consequence, by a simple calculation.   
\end{proof}

We remark here that there exist alternative estimators such as those from \cite{minsker2015geometric,whitehouse2024mean} that can be computed in time usually polynomial in the sample size (generally requiring higher computational effort than the median-of-means estimator described in \Cref{algo:MoM}\footnote{For instance, \cite{minsker2015geometric} requires solving the convex program $\min_{\vx \in \cE}\sum_{i=1}^m \norm{\vx - \vetah_i}$, where, as in \Cref{algo:MoM}, $\vetah_i$ are empirical means of subgroups of samples.}) but that allow for better universal constants in the sample size $\cO(\frac{\ln(1/\delta)}{\epsilon^2})$ (with the same dependence on $\epsilon, \delta$). Any of those estimators can be used in place of median-of-means in our results.

\section{Variance-Reduced Gradual Halpern Algorithm}\label{sec:GHAL}

In this section, we provide results on the convergence of a variant of GHAL \cite{diakonikolas2025pushing}, under assumptions we later show are satisfied by our variance-reduced estimator $\mTt_{k-1, j}.$ The variant of GHAL that we use is slightly different than the algorithms in \cite{diakonikolas2025pushing}, for two main reasons that have to do with the stochastic nature of the algorithm: (1) we do not want to assume that $\mTt$ maps a convex set of \emph{known diameter} to itself (so \cite[GHAL]{diakonikolas2025pushing} is not directly applicable), and (2) even if we proved that the algorithm iterates stay within a bounded set with high probability, we do not want to make step size updates based on algorithm iterates, which themselves are random variables (so \cite[AdaGHAL]{diakonikolas2025pushing} is not directly applicable). The considered variant of GHAL that incorporates variance reduction (VR-GHAL) is summarized in \Cref{algo:GHAL}. We highlight here that the only problem parameter that the algorithm needs to know (apart from target $\delta \in (0,1)$ and the number of epochs $K$ that, as is standard, are part of the input) is the Lipschitz constant $\gamma \in (0, 1].$ We also stress that the algorithm guarantee is \emph{anytime}: instead of fixing the target value of $\eps$, the algorithm can be run for any number of epochs $k$, after which error $\eps = \eps(k)$ is attained with probability at least $1-\delta.$ %

\begin{algorithm}
\caption{Variance-Reduced Gradual Halpern Algorithm (VR-GHAL)}\label{algo:GHAL}
    \begin{algorithmic}[1]
    \Statex \textbf{Input}: $\vx_0 \in \cE,$ $\beta \in (0, 1),$ $\delta \in (0, 1)$, $\gammab \in [\gamma ,1]$, $K \geq 1$
     \Statex \textbf{Initialization}: $\delta_0 = (1-\beta) \delta/\beta$, $\epsilon_0 = 1,$ $\lambdab_0 = 1,$ $k \leftarrow 0,\, \vxh_0 \leftarrow \vx_0$, $\mTt_{0, 0} = \mathrm{MoME}(\beta\epsilon_0/2, \beta\delta_0/2, \vxh_{0})$
     \For{$k=1$ to $K$}
     \State $\epsilon_k \leftarrow \beta \epsilon_{k-1}$
     \State $\lambdab_k \leftarrow \beta\lambdab_{k-1},$ $\lambda_k = \lambdab_k/(1+\lambdab_k)$
     \State $\delta_k = \beta \delta_{k-1}$
     \State $\vy_0 = \vxh_{k-1}$, $J_k = \lceil\ln(\beta)/\ln((1-\lambda_k)\gammab)\rceil$
     \For{$j = 0$ to $J_k - 1$} 
     \State $\vy_{j+1} = \lambda_k \vy_0 + (1-\lambda_k)\mTt_{k-1, j}$
     \State Compute $\mTt_{k-1,j+1}$, a stochastic estimate of $\mT(\vy_{j+1})$
     \EndFor
     \State $\vxh_k = \vy_{J_k}$
     \State $\mTt_{k, 0} = \mathrm{MoME}(\epsilon_{k+1}/2, \delta_{k+1}/2, \vxh_{k})$
     \EndFor
     \State\Return $\vxh_{k}$
    \end{algorithmic}
\end{algorithm}

Below, we state and prove a simple lemma that is useful for tracking the algorithm's per-epoch progress. We defer the complete convergence analysis to \Cref{sec:VR}, where we fully specify the MoME call in \Cref{algo:GHAL} and where  our variance-reduced estimator $\mTt_{k-1, j}$ is introduced and its properties are analyzed in detail. 

\begin{lemma}\label{lemma:basic-GHAL-inequalities}
    Let $\mT:\cE\to \cE$ be a $\gamma$-Lipschitz operator with $\gamma \in (0, 1]$. Fix any outer iteration $k \geq 1$ in \Cref{algo:GHAL} and consider the algorithm states inside it. Let $\gammab \in [\gamma, 1].$ Suppose that for all indices $j \in [J_k]$ of the for loop, we have $\|\mTt_{k-1, j} - \mTt_{k-1, j-1}\| \leq \gammab \norm{\vy_j - \vy_{j-1}}$ and that there exists a sequence $\teps_{k, j}$ such that $\|\mTt_{k-1, j} - \mT(\vy_j)\| \leq \teps_{k, j},$ for all $j \in \{0, 1, \dots, J_k\}.$ Then both of the following two statements hold:
    \begin{align*}
        \|\mTt_{k-1, J_k} - \vxh_k\| &\leq \beta \|\mTt_{k-1, 0} - \vxh_{k-1}\| + \frac{\lambda_k}{1-\lambda_k}\|\vxh_k - \vxh_{k-1}\|, \text{ and}\\
        \|\vxh_k - \vx_*\| &\leq \|\vxh_{k-1} - \vx_*\| + \teps_k,
    \end{align*}
    where $\teps_k := (1-\lambda_k)\sum_{j=0}^{J_k - 1}((1-\lambda_k)\gamma)^{J_k - j -1}\teps_{k, j}.$ In particular, denoting $\lambdab_k = \frac{\lambda_k}{1-\lambda_k},$ we have
    \begin{equation}\label{eq:GHAL-guarantee}
        \|\mT(\vxh_k) - \vxh_k\| \leq \beta \|\mTt_{k-1, 0} - \vxh_{k-1}\| + \teps_{k, J_k} + \lambdab_k\teps_k + 2\lambdab_k\Big(\|\vx_0 - \vx_*\| + \sum_{i=1}^{k-1}\teps_i\Big). 
    \end{equation}
\end{lemma}
\begin{proof}
    The first claim is a direct consequence of \cite[Lemma 1]{diakonikolas2025pushing} and we provide the proof for completeness. In particular, by the definition of $\vy_j$ and the lemma assumption, we have that $\|\vy_1 - \vy_0\| = (1-\lambda_k)\|\mTt_{k-1, 0} - \vy_0\| = (1-\lambda_k)\|\mTt_{k-1, 0} - \vxh_{k-1}\|$, while for $j \geq 1$:
    \begin{align}
        \|\vy_{j+1} - \vy_j\| &= (1-\lambda_k)\|\mTt_{k-1, j} - \mTt_{k-1, j-1}\|\notag\\
        &\leq (1-\lambda_k)\gammab\|\vy_j - \vy_{j-1}\|\notag\\
        &\leq (1-\lambda_k)^{j+1} \gammab^j\|\mTt_{k-1, 0} - \vxh_{k-1}\|,\notag
    \end{align}
    where the first inequality holds by the lemma assumption, and the second inequality is obtained by recursively applying the first one. 
    In particular, $\|\vy_{J_{k+1}} - \vy_{J_k}\| \leq (1-\lambda_k)^{J_k+1} \gammab^{J_k}\|\mTt_{k-1, 0} - \vxh_{k-1}\| \leq (1-\lambda_k)\beta \|\mTt_{k-1, 0} - \vxh_{k-1}\|,$ where the last inequality follows by the choice of $J_k$ in the algorithm statement, and for the analysis only we define $\vy_{J_k + 1} = \lambda_k \vy_0 + (1-\lambda_k)\mTt_{k-1, J_k}$. Further, subtracting $\vy_{J_k}$ from both sides of $\vy_{J_{k+1}} = \lambda_k \vy_0 + (1-\lambda_k)\mTt_{k-1, J_k},$ after elementary algebra, we have $\|\mTt_{k-1, J_k} - \vy_{J_k}\| \leq \frac{1}{1-\lambda_k}\|\vy_{J_{k+1} - \vy_{J_k}}\| + \frac{\lambda_k}{1-\lambda_k}\|\vy_{J_k}- \vy_0\|,$ which leads to the first claim after using $\|\vy_{J_{k+1}} - \vy_{J_k}\| \leq (1-\lambda_k)\beta \|\mTt_{k-1, 0} - \vxh_{k-1}\|$ that we have already derived and $\vy_{J_k} = \vxh_k,$ $\vy_0 = \vxh_{k-1}.$ Note that this first claim did not make use of any probabilistic assumptions.

    For the second claim, assume that $\|\mTt_{k-1, j} - \mT(\vy_j)\| \leq \teps_{k, j},$ for all $j \in\{0, 1, \dots, J_k\},$ which holds by the lemma statement. By the definition of $\vy_{j+1}$ and triangle inequality, we have that for all $j \geq 1,$
    \begin{align}
        \|\vy_{j} - \vx_*\| &\leq \lambda_k\|\vy_0 - \vx_*\| + (1-\lambda_k)\|\mTt_{k-1, j-1} - \mT(\vy_{j-1})\| + (1-\lambda_k)\|\mT(\vy_{j-1}) - \mT(\vx_*)\|\notag\\
        &\leq \lambda_k\|\vy_0 - \vx_*\| + (1-\lambda_k)\teps_{k, j-1} + (1-\lambda_k)\gamma\|\vy_{j-1} - \vx_*\| \label{eq:dist-to-opt-recursive},
    \end{align}
    where we used $\vx_* = \mT(\vx_*),$ $\|\mTt_{k-1, j-1} - \mT(\vy_{j-1})\| \leq \teps_{k, j-1}$, and $\|\mT(\vy_{j-1}) - \mT(\vx_*)\| \leq \gamma \|\vy_{j-1} - \vx_*\| \leq \|\vy_{j-1} - \vx_*\|,$ all of which hold by assumption. Let $\rho_k := (1-\lambda_k)\gamma$. Observe that \eqref{eq:dist-to-opt-recursive} is a recursive inequality for $\|\vy_{j} - \vx_*\|$ and unroll the recursion to get:
    \begin{equation}\notag
        \norm{\vy_j - \vx_*} \leq \big(\rho_k^{j} + \lambda_k\sum_{i=0}^{j-1}\rho_k^{i}\big)\norm{\vy_0 - \vx_*} + (1-\lambda_k)\sum_{i=0}^{j-1}\rho_k^{j-i-1} \teps_{k, i}\big).
    \end{equation}
    Since $\gamma \leq 1,$ we have $\rho_k \leq 1 - \lambda_k$, and so $\rho_k^{j} \leq (1-\lambda_k)^{j}$ and $\sum_{i=0}^{j-1}\rho_k^{i} \leq \sum_{i=0}^{j-1}(1-\lambda_k)^{i} \leq \frac{1 - (1-\lambda_k)^j}{\lambda_k},$ so the term multiplying $\norm{\vy_0 - \vx_*}$ is at most one.   
    It remains to set $j = J_k$, recall that $\vxh_k = \vy_{J_k}$, $\vy_0 = \vxh_{k-1}$, and use the definition of $\teps_k$.  
\end{proof}

%

\section{Clipped Difference Operator and Convergence of VR-GHAL}\label{sec:VR}

In this section, we introduce our clipped variance-reduced estimator $\mTt(\vx)$ and analyze its properties that determine our stochastic oracle complexity results. One of the core contributions is an estimator that is based on clipping the iterate-difference stochastic oracle $\Delta(\vx, \vy) = \vtau(\vx; \vxi_1) - \vtau(\vy; \vxi_2).$ To our knowledge, while clipped (or trimmed) mean estimators based on constant-value or fixed-schedule clipping of the random variable/stochastic oracle $\vtau(\vx; \vxi)$ itself are common in the literature \cite{gorbunov2020clipping,nguyen2023clippedsgd,cutkosky2021highprob,zhang2022parameter,gorbunov2024composite,sadiev2023highprob}, we are not aware of estimators that explicitly utilize the Lipschitzness of the expected operator to perform the clipping dependent on the iterate-distance bound implied by Lipschitzness while yielding an implementable algorithm. 

We begin the section by providing general, auxiliary results for the properties of the introduced \emph{clipped difference} operator (\Cref{sec:clipping}). We then define our variance-reduced estimator and analyze its properties in \Cref{sec:VR+master-thm}, leading to our Master Theorem (\Cref{thm:master}). In \Cref{sec:gen-bnd-var}--\Cref{sec:other-models} we then apply the Master Theorem to the different stochastic oracle models (bounded variance only, Lipschitz in expectation, other models) to yield quantitative sample complexity results. 

\subsection{Clipped Difference Operator}\label{sec:clipping}

A basic ingredient of our stochastic estimator is a truncation (or clipping) operator, which differs from the use of ``gradient clipping'' in the recent literature in two important aspects. First, unlike the recent literature on ``gradient clipping,'' we do not perform the clipping operation directly on the operator (or on $\mTt(\vx) - \vx$), but on a \emph{difference} operator $\Delta(\vx, \vy)$ that satisfies $\E[\Delta(\vx, \vy)] = \mT(\vx) - \mT(\vy).$ For instance, an operator defined by
\begin{equation}\label{eq:diff-op-ex}
    \Delta(\vx, \vy) := \vtau(\vx; \vxi_1) - \vtau(\vy; \vxi_2),
\end{equation}
or by an empirical average of such operators---where we take $\vxi_1 = \vxi_2 \sim \cD$ whenever \Cref{asspt:multi-query-oracle} is made and otherwise draw $\vxi_1, \vxi_2$ i.i.d.\ from $\cD$---would satisfy such a requirement. Second, we crucially leverage the Lipschitzness of $\mT$ to determine how to clip the difference operator $\Delta.$ In particular, the clipped value of the $\Delta$ operator is defined by
\begin{equation}\label{eq:clipped-diff-op}
    \cldel(\vx, \vy) = \min\Big\{1, \, \frac{\gammab\|\vx - \vy\|}{\|\Delta(\vx, \vy)\|}\Big\}\Delta(\vx, \vy),
\end{equation}
where we interpret the above equation as assigning $\cldel(\vx, \vy) = \vzero$ whenever $\Delta(\vx, \vy) = \vzero.$

The following lemma summarizes the basic properties of this clipped difference estimator that will be used in the subsequent analysis. The proof is provided in \Cref{appx:omitted-clipping}. For similar properties of the constant-clipped gradient estimator, see, e.g., \cite[Lemma 2.1]{nguyen2023clippedsgd}.

\begin{restatable}{lemma}{lemmaclipping}\label{lemma:clipped-op-properties}
    Given a stochastic difference operator $\Delta(\vx, \vy)$ that satisfies $\|\E[\Delta(\vx, \vy)]\| \leq \gamma \|\vx - \vy\|$ for some $\gamma > 0$ and let the clipped difference operator be defined by \eqref{eq:clipped-diff-op} for some $\gammab \geq \gamma$. Given any $\vx, \vy,$ suppose further that $\E[\|\Delta(\vx, \vy) - \E[\Delta(\vx, \vy)]\|^2] \leq \sigma_{\vx \vy}^2 < \infty$. Then:
    \begin{enumerate}
        \item $\|\cldel(\vx, \vy) - \E[\cldel(\vx, \vy)]\| \leq 2\gammab \|\vx - \vy\|;$
        \item $\|\E[\cldel(\vx, \vy)] - \E[\Delta(\vx, \vy)]\| \leq \sigma_{\vx\vy}$; and%
        \item $\E[\|\cldel(\vx, \vy) - \E[\cldel(\vx, \vy)]\|^2] \leq  9 \sigma_{\vx\vy}^2.$
    \end{enumerate}
\end{restatable}

\subsection{Variance-Reduced Estimator and Master Theorem}
\label{sec:VR+master-thm}

Our stochastic variance-reduced estimator $\mTt^{(k)}_j$ is defined recursively over each epoch, as described below. Given a fixed epoch $k \geq 1,$ let $\vy_j$ denote iterates of the inner for loop (see \Cref{algo:GHAL}). For $j \geq 1,$ define the minibatch difference operator $\Delta_{k, j}$ as
\begin{equation}\label{eq:delta-kj-def}
    \Delta_{k, j}^{(\ell)}:= \frac{1}{m_{k, j}}\sum_{i=1}^{m_{k, j}}(\vtau(\vy_j, \vxi_{i}^{(1)}) - \vtau(\vy_{j-1}, \vxi_{i}^{(2)})),
\end{equation}
where, as described in \Cref{sec:clipping}, $\vxi_{i}^{(1)} = \vxi_i^{(2)}$ under \Cref{asspt:multi-query-oracle} and are otherwise i.i.d.; all samples $\vxi_1^{(\iota)}, \vxi_2^{(\iota)}, \dots, \vxi_{m_{k, \iota}}^{(\iota)}$ for $\iota \in \{1, 2\}$ are i.i.d. Notice that $\Delta_{k, j}^{(\ell)}$ is itself a random variable. We draw $n_k$ i.i.d.\ copies $\Delta_{k, j}^{(\ell)},$ $\ell \in [n_k]$, for each $j \geq 1$ and define the stochastic variance-reduced operator $\mTt_{k, j}$ as:
\begin{equation}\label{eq:VR-estimator}
    \mTt_{k-1, j} := \begin{cases}
        \mTh_{k-1}(\vy_0; \epsilon_k/2, \delta_{k}/2) = \mTh_{k-1}(\vxh_{k-1}; \epsilon_k/2, \delta_{k}/2), &\text{ if } j = 0,\\
        \mTt_{k-1, j-1} + \frac{1}{n_k}\sum_{\ell =1}^{n_k}\cl\Delta_{k, j}^{(\ell)}, &\text{ if } j \geq 1.
    \end{cases}
\end{equation}
Above, $\mTh_{k-1}$ can be any high-probability estimator for $\mT(\vxh_{k-1})$ that ensures $\|\mT(\vxh_{k-1}) - \mTh_{k-1}\| \leq (\epsilon_k/2)\sqrt{\kappa_\cE}\sigma $ with probability at least $1-\delta_{k}/2.$ For concreteness, let $\mTh_{k} :=\mathrm{MoME}(\epsilon_{k+1}/2, \delta_{k+1}/2, \vxh_{k})$, as in \Cref{algo:GHAL}. In the algorithm, $\mathrm{MoME}(\epsilon_{k+1}/2, \delta_{k+1}/2, \vxh_{k})$ refers to the call of the median-of-means estimator discussed in \Cref{sec:mom} for stochastic estimation on $\mT(\vxh_k)$ based on calls to $\vtau(\vxh_k; \vxi_i)$, $i \in [N_k],$ for failure probability $\delta_{k+1}/2$ and target error $\|\mTt_{k, 0} - \mT(\vxh_k)\| \leq (\eps_{k+1}/2)\sqrt{\kappa_\cE}\sigma.$ Based on \Cref{prop:MoME-final}, $N_k = \cO(\frac{\ln({1}/{\delta_k})}{\eps_k^2})$ oracle calls to $\vtau$ suffice.  %

In the rest of the section, we establish properties of the estimator \eqref{eq:VR-estimator} that are needed for carrying out our analysis. The first property used in \Cref{algo:GHAL}---that $ \|{\mTt_{k-1, j} - \mTt_{k-1, j-1}}\| \leq \gammab\norm{\vy_{j} - \vy_{j-1}}$ for $j \geq 1$---is immediate by its definition, since $\mTt_{k-1, j} - \mTt_{k-1, {j-1}} = \frac{1}{n_k}\sum_{\ell =1}^{n_k}\cl\Delta_{k, j}^{(\ell)}$, where each $\cl\Delta_{k, j}^{(\ell)}$ is a clipped difference operator (thus it satisfies $\|\cl\Delta_{k, j}^{(\ell)}\| \leq \gammab\norm{\vy_j - \vy_{j-1}}$). 

%

The less immediate property concerns bounding $\|\mTt_{k-1, j} - \mT(\vy_j)\|$ with high probability for $j \geq 1$ (for $j = 0$, the result is immediate, as $\mTt_{k-1, 0}$ is a high-probability estimator of $\mT(\vy_0)$ by its definition), which requires more technical effort.  This result is summarized in the following proposition.

\begin{proposition}\label{prop:variance-of-estimator}
    Let $(\cE, \norm{\cdot})$ be a quadratically smoothable space (\Cref{asspt:Banach}). Fix $k \geq 1$ and consider the $k^{\mathrm{th}}$ epoch of GHAL (\Cref{algo:GHAL}). Let $\cF_{k, i}^{(\ell)}$ denote the natural filtration, containing all algorithm randomness up to and including epoch $k$,  inner iteration $i$, and random draw $\ell$ ($i = 0$ is the beginning of the epoch, and $\ell=0$ denotes the state before any samples are drawn). Suppose that $\Rt_{k-1} := \|\mTt_{k-1, 0} - \vy_0\| < \infty$ and that $\|\mTt_{k-1, 0} - \mT(\vy_{0})\| \leq (\eps_k/2)\sqrt{\kappa_\cE}\sigma $. Let $\E[\|\Delta_{k, i}^{(\ell)} - \E[\Delta_{k, i}^{(\ell)}|\cF_{k, i}^{(0)}]\|^2|\cF_{k, i}^{(0)}] \leq \sigma_{k, i}^2 < \infty$ a.s., $\forall \ell \in [n_k]$, where $\sigma_{k, i}^2$ is $\cF_{k, 0}^{(0)}$-measurable. Then,  conditionally on $\cF_{k, 0}^{(0)}$, with probability $1-\delta_k/2$, the following holds simultaneously for all $j \in  [J_k]$:
    \begin{equation}\notag
    \begin{aligned}
        \|\mTt_{k-1, j} - \mT(\vy_j)\| \leq (\eps_k/2)\sqrt{\kappa_\cE}\sigma + \sum_{i=1}^{j} \sigma_{k, i} + 3\sqrt{\frac{2\kappa_\cE \ln(4/\delta_k)}{n_k}\sum_{i=1}^{J_k} \sigma_{k, i}^2} + \frac{2C(1-\lambda_k)\gammab \Rt_{k-1} \ln\big({4}/{\delta_k}\big)}{3 n_k}. 
    \end{aligned}
    \end{equation}
\end{proposition}
\begin{proof}
    As a first step, we use the recursive nature of the definition of $\mTt_{k-1, j}$ to write it in a form that includes a martingale (for which we can use \Cref{prop:Banach-concentration}) and other terms that can be bounded using properties of clipped operators stated in \Cref{lemma:clipped-op-properties}. In particular, observe that we can write for $j \geq 1$
    \begin{align}
        \mTt_{k-1, j} - \mT(\vy_j) &= \mTt_{k-1, j-1} + \frac{1}{n_k}\sum_{\ell = 1}^{n_k}\cldel_{k, j}^{(\ell)} - \mT(\vy_j)\notag\\
        &= \mTt_{k-1, j-1} - \mT(\vy_{j-1}) + \frac{1}{n_k}\sum_{\ell = 1}^{n_k}\cldel_{k, j}^{(\ell)} - (\mT(\vy_j) - \mT(\vy_{j-1}))\notag\\
        &= \mTt_{k-1, 0} - \mT(\vy_{0}) + \frac{1}{n_k}\sum_{i=1}^j \sum_{\ell =1}^{n_k}\big(\cldel_{k, i}^{(\ell)} - (\mT(\vy_i) - \mT(\vy_{i-1}))\big)\notag\\
        &= \mTt_{k-1, 0} - \mT(\vy_{0}) + \frac{1}{n_k}\sum_{i=1}^j  \sum_{\ell = 1}^{n_k}\big(\cldel_{k, i}^{(\ell)} - \E[\cldel_{k, i}^{(\ell)} |\cF_{k, i}^{(0)}]\big) \notag\\
         &\quad+ \sum_{i'=1}^j\frac{1}{n_k}\sum_{\ell = 1}^{n_k}\big(\E[\cldel_{k, i'}^{(\ell)}|\cF_{k, i'}^{(0)}] - \E[\Delta_{k, i'}^{(\ell)}|\cF_{k, i'}^{(0)}]\big), \notag
    \end{align}
    where the third equality comes from unrolling the recursion from the second equality and the last equality uses $\mT(\vy_i) - \mT(\vy_{i-1}) = \E[\Delta_{k, i}^{(\ell)}|\cF_{k, i}^{(0)}],$ which holds by the definition of $\Delta_{k, i}^{(\ell)}$. Thus, by the triangle inequality,
    \begin{equation}\label{eq:vr-bnd-1}
    \begin{aligned}
        \|\mTt_{k-1, j} - \mT(\vy_j)\| \leq \; & \|\mTt_{k-1, 0} - \mT(\vy_{0})\|
        + \frac{1}{n_k}\Big\| \sum_{i=1}^j \sum_{\ell = 1}^{n_k} \big(\cldel_{k, i}^{(\ell)} - \E[\cldel_{k, i}^{(\ell)}|\cF_{k, i}^{(0)}]\big) \Big\|\\
        &+ \frac{1}{n_k}\sum_{\ell = 1}^{n_k}\sum_{i'=1}^j\big\|\E[\cldel_{k, i'}^{(\ell)}|\cF_{k, i'}^{(0)}] - \E[\Delta_{k, i'}^{(\ell)}|\cF_{k, i'}^{(0)}]\big\|.
    \end{aligned}
    \end{equation}
    To prove the claimed bound, we bound each of the terms on the right-hand side of \eqref{eq:vr-bnd-1} separately. First, by the proposition statement, we have that 
    \begin{equation}\label{eq:vr-var-1}
        \|\mTt_{k-1, 0} - \mT(\vy_{0})\| \leq (\eps_k/2)\sqrt{\kappa_\cE}\sigma .
    \end{equation} 
    The last summation contains conditional expectations that are $\cF_{k, i'}^{(0)}$-measurable, thus we can bound each summation term a.s.\ using \Cref{lemma:clipped-op-properties}, Part 2, applied conditionally on $\cF_{k, i'}^{(0)}$ by $\big\|\E[\cldel_{k, i'}^{(\ell)}|\cF_{k, i'}^{(0)}] - \E[\Delta_{k, i'}^{(\ell)}|\cF_{k, i'}^{(0)}]\big\| \leq \sigma_{k, i'}.$ 
Thus, we can conclude that:
\begin{equation}\label{eq:dif-of-exps-sum}
    \frac{1}{n_k}\sum_{\ell = 1}^{n_k}\sum_{i'=1}^j\big\|\E[\cldel_{k, i'}^{(\ell)}|\cF_{k, i'}^{(0)}] - \E[\Delta_{k, i'}^{(\ell)}|\cF_{k, i'}^{(0)}]\big\| \leq %
    \sum_{i=1}^j \sigma_{k, i}. %
\end{equation}

    To apply  \Cref{prop:Banach-concentration}, we need to verify its assumptions. Recall that $\cF_{k, i}^{(\ell)}$ denotes the natural filtration, containing all algorithm randomness up to and including epoch $k$,  inner iteration $i$, and random draw $\ell$. Note that $\cF_{k, i-1}^{(n_k)} = \cF_{k, i}^{(0)}.$  Then, $\E[\cldel_{k, i}^{(1)} - \E[\cldel_{k, i}^{(1)}|\cF_{k, i}^{(0)}]|\cF_{k, i-1}^{(n_k)}] = \vzero$ and $\E[\cldel_{k, i}^{(\ell)} - \E[\cldel_{k, i}^{(\ell)}|\cF_{k, i}^{(0)}]|\cF_{k, i}^{(\ell - 1)}] = \vzero$ for $\ell \geq 2$ is immediate. Thus, using indexing $\iota = (i-1)n_k + \ell$, 
    \[\cM_\iota := \sum_{i'=1}^{i-1}\sum_{\ell' = 1}^{n_k} \big(\cldel_{k, i'}^{(\ell')} - \E[\cldel_{k, i'}^{(\ell')}|\cF_{k, i'}^{(0)}]\big) + \sum_{\ell'' = 1}^{\ell} \big(\cldel_{k, i}^{(\ell'')} - \E[\cldel_{k, i}^{(\ell'')}|\cF_{k, i}^{(0)}]\big)\] 
    defines a martingale. 
    Further, by \Cref{lemma:clipped-op-properties}, Part 1, for all $\ell \in [n_k]$:
    \begin{align}
        \|\cldel_{k, i}^{(\ell)} - \E[\cldel_{k, i}^{(\ell)}|\cF_{k, i}^{(0)}]\| &\leq 2\gammab\|\vy_i - \vy_{i-1}\| \notag\\
        &\leq 2\gammab \|\mTt_{k-1, 0} - \vy_0\|(1-\lambda_k)^i\notag\\
        &= 2\gammab \Rt_{k-1}(1-\lambda_k)^i =: b_i,\notag%
    \end{align}
    where in the second inequality we dropped the factor $\gamma^i$, since $\gamma \leq 1.$ 
    It remains to bound the variance, which follows by \Cref{lemma:clipped-op-properties}, Part 3:
     \begin{align}\notag%
       \E[\|\cldel_{k, i}^{(\ell)} - \E[\cldel_{k, i}^{(\ell)}|\cF_{k, i}^{(0)}]\|^2|\cF_{k, i}^{(0)}] &\leq 9\sigma_{k, i}^2.
    \end{align}
    Thus, \Cref{prop:Banach-concentration} implies that with probability $1- \delta_k/2,$ for all $j \in [J_k]$ simultaneously,
    \begin{equation}\label{eq:vr-martingale}
       \norm{\frac{1}{n_k}\sum_{i=1}^j  \sum_{\ell = 1}^{n_k}\big(\cldel_{k, i}^{(\ell)} - \E[\cldel_{k, i}^{(\ell)}|\cF_{k, i}^{(0)}]\big)} \leq 3\sqrt{\frac{2\kappa_\cE \ln(4/\delta_k)}{n_k}\sum_{i=1}^{J_k} \sigma_{k, i}^2} + \frac{2C(1-\lambda_k)\gammab \Rt_{k-1} \ln\big({4}/{\delta_k}\big)}{3 n_k}.
    \end{equation}
    To complete the proof, it remains to combine \eqref{eq:vr-bnd-1} with \eqref{eq:vr-var-1}, \eqref{eq:dif-of-exps-sum}, and \eqref{eq:vr-martingale}. 
\end{proof}

Our main result, \Cref{thm:master}, is based on the following condition on the variance of clipped difference estimators, stated in \Cref{asspt:variance-minibatch}. To obtain specific results on sample complexity, in the remaining subsections, we then provide sufficient bounds on the sample sizes $m_{k, j}$ for which \Cref{asspt:variance-minibatch} holds. 

\begin{assumption}\label{asspt:variance-minibatch}
    The per-iteration sample sizes $m_{k, j}$'s are chosen such that 
    \begin{equation}\label{eq:per-iteration-variance}
        \sum_{i=1}^{J_k} \sigma_{k, i} + 3\sqrt{{2\kappa_\cE\ln(4/\delta_k)}\sum_{i=1}^{J_k} \sigma_{k, i}^2} \leq \eps_k (a_0 + a_1\sqrt{k}), \quad \forall k \geq 1,
    \end{equation}
    where $a_0, a_1$ are potentially dependent on (fixed) problem parameters $\sigma, \delta, \kappa_\cE$.
\end{assumption}

\begin{theorem}[Master Theorem]\label{thm:master}
    Let $(\cE, \norm{\cdot})$ be a quadratically smoothable space (\Cref{asspt:Banach}). Given  input parameters $\eps > 0, \delta \in (0, 1)$ and access to a stochastic oracle $\vtau(\vx; \vxi),$ $\vxi \sim \cD,$ for a $\gamma$-Lipschitz operator $\mT$ with $\gamma \leq 1$ that satisfies \eqref{eq:stochastic-opt-assmpts}, consider VR-GHAL (\Cref{algo:GHAL}) iterations. Let $n_k = \max\big\{1, \lceil \frac{2(1-\lambda_k)\gammab \Rt_{k-1} \ln({4}/{\delta_k})}{3 \epsilon_k} \rceil\big\}$ and suppose that \Cref{asspt:variance-minibatch} holds.   
    Let $K$ satisfy
\begin{equation}\label{eq:generic-k}
        \begin{aligned}
            K \geq \Big\lceil \max_{p \in \{0, 1, 2, 5/2\}}\Big\{L_p + p\log_{1/\beta}L_p + \cO_\beta(1) \Big\} \Big\rceil,
        \end{aligned}
    \end{equation}
where $\cO_\beta(1)$ is a universal constant whenever $\beta$ is a constant away from one; $L_p = \max\{1, \log_{1/\beta}(8A_p/\epsilon)\},$ and constants $A_p$ are defined as
    \begin{equation}\label{eq:C-params}
    \begin{aligned}
        A_0 &= \|\mT(\vxh_0) - \vxh_0\| + \beta \eps_0 \sqrt{\kappa_\cE}\sigma /2, \;\; && A_1 = 2\lambdab_0 \norm{\vx_* - \vx_0} + \beta \eps_0\sqrt{\kappa_\cE}\sigma /2,\\
        A_2 &= \eps_0(C + \sqrt{\kappa_\cE}\sigma/2 + a_0),\;\; && A_{5/2} = (4\sqrt{2}/5)\eps_0 a_1. 
    \end{aligned}
\end{equation}

Then, with probability at least $1-\delta$, the output $\vxh_K$ of \Cref{algo:GHAL} satisfies
$\|\mT(\vxh_K) - \vxh_K\| \leq \epsilon$. Moreover, the event on which the guarantee holds may be chosen independently of $K$; hence, on an event of probability at least $1-\delta$, the bound holds for every epoch $k$ satisfying the inequality in \eqref{eq:generic-k}.
\end{theorem}
\begin{proof}
    For each $k \geq 0,$ define $\Rt_k := \|\mTt_{k, 0} - \vxh_k\|.$ 
    Let $\teps_{k, j}:= \|\mTt_{k-1, j} - \mT(\vy_j)\|.$ Define events $E_k$ by
    \begin{equation}
    \begin{aligned}
        E_0 &:= \{\teps_{1, 0} \leq (\eps_1/2)\sqrt{\kappa_\cE}\sigma \},\\
        E_k &:= E_{k-1} \cap \big\{\teps_{k+1, 0} \leq (\eps_{k+1}/2)\sqrt{\kappa_\cE}\sigma \text{ and } \teps_{k, j} \leq \eps_k(C + \sqrt{\kappa_\cE}\sigma/2 + a_0 + a_1\sqrt{k}), \forall j \in [J_{k}]\big\}, k \geq 1
    \end{aligned} 
    \end{equation}
Observe that if event $E_k$ is true, then all events $E_i,$ $i \leq k -1,$ must be true as well, by definition. To establish a bound on   $\Rt_k := \|\mTt_{k, 0} - \vxh_{k}\|$, which is needed for bounding the probability of event $E_{k+1}$ (as $\teps_{k+1, j}$'s are bounded dependent on $\Rt_k$), we carry out an inductive argument sequentially conditioning on events $E_k.$ Along the way, we inductively argue $\prob[E_k] \geq 1 - \sum_{i=1}^{k}\delta_i - \delta_{k+1}/2.$ Let us immediately notice here that $\prob[E_k] \geq 1 - \sum_{i=1}^{k}\delta_i - \delta_{k+1}/2$ immediately implies $\prob[E_k] \geq 1- \delta$ no matter how large $k$ is, since
\begin{equation}\label{eq:fail-prob}
    \sum_{i=1}^k \delta_i \leq \sum_{i=1}^\infty\delta_i = (1-\beta)\delta\sum_{i=0}^{\infty}\beta^i \leq  \delta. 
\end{equation}
Thus, the claimed success probability of the algorithm immediately follows from the inductive proof below.

Our proof by induction will establish:
\begin{equation}\label{eq:bnd-on-Rtk}
    \Rt_k \leq \beta^k (A_0 + A_1 k  + A_2 (k+2)^2  + A_{5/2} (k+2)^{5/2}),
\end{equation}
where $A_0, A_1, A_2,$ and $A_{5/2}$ are defined as in the theorem statement.

For the base case, it is immediate (from the definition of $\mTt_{0, 0}$) that event $E_0$ holds with probability $1 - \delta_1/2$, as the guarantee holds from the first call to the MoME in \Cref{algo:GHAL}. Under the event $E_0,$ we have
\begin{equation}
    \Rt_0 := \|\mTt_{0, 0} - \vxh_0\| \leq \|\mT(\vxh_0) - \vxh_0\| + (\eps_1/2)\sqrt{\kappa_\cE}\sigma = A_0. 
\end{equation}

Suppose now that for $k \geq 1$, $E_{k-1}$ holds with probability at least $1 - \sum_{i=1}^{k-1}\delta_i - \delta_{k}/2$, and, under $E_{k-1},$ we have
\begin{equation}\label{eq:ind-hypothesis}
    \Rt_{k-1} \leq \beta^{k-1} (A_0 + A_1 (k-1)  + A_2 (k+1)^2  + A_{5/2} (k+1)^{5/2}).
\end{equation}
Then, by \Cref{prop:variance-of-estimator}, under $E_{k-1},$ simultaneously for all $j \in [J_k]$, with probability $1-\delta_k/2$, we have that %
    \begin{equation}\label{eq:teps-kj}
    \begin{aligned}
       \teps_{k, j} :=  \|\mTt_{k-1, j} - \mT(\vy_j)\| \leq \frac{\eps_k \sqrt{\kappa_\cE}\sigma }{2} + \sum_{i=1}^j \sigma_{k, i} + 3\sqrt{2\kappa_\cE \ln(4/\delta_k)\sum_{i=1}^j \sigma_{k, i}^2} + \frac{2C(1-\lambda_k)\gammab \Rt_{k-1} \ln({4}/{\delta_k})}{3 n_k}. 
    \end{aligned}
    \end{equation}
For $n_k = \lceil \frac{2(1-\lambda_k)\gammab \Rt_{k-1} \ln({4}/{\delta_k})}{3 \epsilon_k} \rceil$, we have that $\frac{2C(1-\lambda_k)\gammab \Rt_{k-1} \ln({4}/{\delta_k})}{3 n_k} \leq C\epsilon_k.$ Further, by \Cref{lemma:minibatching-variance} and using properties of quadratically smoothable spaces and the definition of $\sigma_{k, j}$, we have that 
\begin{equation}\label{eq:sigmakj-bnd}
    \sigma_{k, j}^2 \leq \frac{\kappa_\cE}{m_{k, j}}\Big(\frac{1}{m_{k, j}}\sum_{i=1}^{m_{k, j}}\E\norm{\vtau(\vy_j, \vxi_i^{(1)}) - \vtau(\vy_{j-1}, \vxi_i^{(2)}) - (\mT(\vy_j) - \mT(\vy_{j-1}))}^2\Big).
\end{equation} 
Because the minimum assumption that we make is that the variance is bounded, each summation term in \eqref{eq:sigmakj-bnd} is bounded, which justifies the statement that we can choose sufficiently large (but finite) $m_{k,i}$'s ensuring $\sigma_{k, i}$'s are suitably small so that $\sum_{i=1}^j \sigma_{k, i} + 3\sqrt{2\kappa_\cE\ln(4/\delta_k)\sum_{i=1}^j \sigma_{k, i}^2} \leq \eps_k (a_0 + a_1 \sqrt{k}).$ 

Thus, under $E_{k-1}$ and $\|\mTt_{k-1, 0} - \vxh_{k-1}\| \leq \Rt_{k-1}$, with probability $1-\delta_k/2,$ we have that for all $j \in [J_k],$ $\teps_{k, j} \leq \eps_k(C + \sqrt{\kappa_\cE}\sigma/2 + a_0 + a_1\sqrt{k}).$ Additionally, by the MoME guarantee, we have that $\|\mTt_{k, 0} - \mT(\vxh_k)\| \leq (\eps_{k+1}/2)\sqrt{\kappa_\cE}\sigma ,$ with probability at least $1 - \delta_{k+1}/2$. Thus, by the union bound, event $E_k$ holds with probability at least $1 - \sum_{i=1}^{k-1}\delta_i - \delta_{k}/2 - \delta_k/2 - \delta_{k+1}/2 = 1 - \sum_{i=1}^{k}\delta_i - \delta_{k+1}/2$, as needed for the inductive claim. 

To complete the inductive proof, it remains to argue that under event $E_k$ and $\|\mTt_{k-1, 0} - \vxh_{k-1}\| \leq \Rt_{k-1},$ we additionally have  $\|\mTt_{k, 0} - \vxh_{k}\| \leq \Rt_{k}.$ Observe that under this hypothesis, \Cref{lemma:basic-GHAL-inequalities} applies, and, additionally, $\|\mTt_{k, 0} - \mT(\vxh_k)\| \leq (\eps_{k+1}/2)\sqrt{\kappa_\cE}\sigma $. Thus:
\begin{equation}\label{eq:recursive-fp-res-bnd}
    \|\mTt_{k, 0} - \vxh_k\| \leq \beta \|\mTt_{k-1, 0} - \vxh_{k-1}\| + (\eps_{k+1}/2)\sqrt{\kappa_\cE}\sigma + \teps_{k, J_k} + \lambdab_k\teps_k + 2\lambdab_k\Big(\|\vx_0 - \vx_*\| + \sum_{i=1}^{k-1}\teps_i\Big),
\end{equation}
where, as derived above, for all $j \in [J_k],$ \[\teps_{k, j} \leq \eps_k(C + \sqrt{\kappa_\cE}\sigma/2 + a_0 + a_1\sqrt{k}),\] and by the notation in \Cref{lemma:basic-GHAL-inequalities}, 
\begin{align}
\teps_k &:= (1-\lambda_k)\sum_{j=0}^{J_k - 1}((1-\lambda_k)\gamma)^{J_k - j -1}\teps_{k, j}\notag\\
&\leq \frac{1-\lambda_k}{\lambda_k}\epsilon_k (C + \sqrt{\kappa_\cE}\sigma/2 + a_0 + a_1 \sqrt{k}), \notag
\end{align}
where we have used $\gamma \leq 1$ to crudely bound $\sum_{j=0}^{J_k - 1}((1-\lambda_k)\gamma)^{J_k - j -1} \leq \frac{1}{\lambda_k}$ and the above bound on $\teps_{k, j}.$ Using that $\eps_k = \eps_0 \beta^k$ and $\lambdab_k = \lambdab_0 \beta^k,$ and recalling the notation $\Rt_k = \|\mTt_{k, 0} - \vxh_k\|,$ we can now simplify \eqref{eq:recursive-fp-res-bnd} as follows:
\begin{align}
    \Rt_k &\leq \beta \Rt_{k-1} + 2\lambdab_0 \beta^k \norm{\vx_0 - \vx_*} + \beta^k \eps_0\Big(\beta\sqrt{\kappa_\cE}\sigma/2 + 2\big(C + \sqrt{\kappa_\cE}\sigma/2 + a_0 + a_1 \sqrt{k}\big)(k+1))\Big)\notag\\
    &= \beta \Rt_{k-1} + \beta^k (A_1 + 2A_2(k+1) + \frac{5}{2\sqrt{2}}A_{5/2}(k+1)\sqrt{k}).\label{eq:Rtk-recursion-1}
\end{align}
Plugging in the inductive hypothesis \eqref{eq:ind-hypothesis}, we now have
\begin{align}
    \Rt_k &\leq \beta^k\big(A_0 + A_1 k + A_2((k+1)^2 + 2(k+1)) + A_{5/2}\Big(\frac{5}{2\sqrt{2}}(k+1)\sqrt{k} + (k+1)^{5/2}\Big)\big)\\
    &\leq \beta^k \big(A_0 + A_1 k + A_2(k+2)^2 + A_{5/2}(k+2)^{5/2}\big),
\end{align}
where in the last line we used $(k+2)^2 - (k+1)^2 = 2k + 3 \geq 2(k+1)$ and $(k+2)^{5/2} - (k+1)^{5/2} = \int_{k+1}^{k+2}\frac{5}{2}t^{3/2}\dd t \geq \frac{5}{2}(k+1)^{3/2} \geq \frac{5}{2\sqrt{2}}(k+1)\sqrt{k}$, completing the proof by induction.

To complete the proof of the theorem, it remains to bound the number of iterations $k$ until $\Rt_k \leq \eps/2.$ To do so, we enforce $\beta^k A_p k^p \leq \eps/8$ for each $p \in\{0, 1, 2, 5/2\}$. Let $L_p := \max\{1, \log_{1/\beta}\frac{8A_p}{\eps}\}.$ Then $\beta^k A_p k^p \leq \eps/8$ for $k \geq L_p + p \log_{1/\beta}L_p + \cO_{\beta}(1)$, where $\cO_\beta(1)$ is a small universal constant as long as $\beta$ is not trivially close to one. This leads to the claimed bound on $k.$ 

Finally, on the event $E_k$,
\begin{equation}\notag
    \|\mT(\vxh_k) - \vxh_k\| \leq \Rt_k + \|\mTt_{k, 0} - \mT(\vxh_k)\| <  \Rt_k + \frac{\epsilon_{k+1}}{2}\sqrt{\kappa_\cE}\sigma. 
\end{equation}
By the choice of $k$ in \eqref{eq:generic-k}, the induction bound gives
$\Rt_k \le \epsilon/2$, while the same choice also implies
$(\epsilon_{k+1}/2)\sqrt{\kappa_\cE}\sigma\le \epsilon/2$.
Thus $\|\mT(\vxh_k) - \vxh_k\|\le \epsilon$.
\end{proof}
Although the bound on $k$ may appear somewhat complicated, the important part is that the leading term is logarithmic (with constant one) in $1/\eps.$ All remaining terms are lower-order (constant or $\log\log$). This is important for bounding the polynomial growth in $1/\eps$ in the sample complexity bounds derived in the upcoming subsections. 
Additionally, observe that since $\Rt_k = \cO(\beta^k \mathrm{poly}(k))$ and $\eps_k = \eps_0 \beta^k,$ we have that $n_k$ is polynomial in $k,$ thus it can only contribute poly-logarithmic factors in the resulting sample complexity bounds. 

For completeness, we also bound the total number of samples utilized in the calls to MoME.

\begin{lemma}\label{lemma:MoME-calls-samples}
    The total number of samples that \Cref{algo:GHAL} uses in calls to MoME (over all epochs) is at most
    \[\widetilde{\cO}\Big(\frac{\ln(1/\delta)\max_{p \in \{0, 1, 2, 5/2\}}A_p^2}{\eps^2}\Big).\]
\end{lemma}
\begin{proof}
    By \Cref{prop:MoME-final}, a call to MoME in epoch $k$ costs $\cO(\frac{\ln(1/\delta_k)}{\eps_k^2}) = \cO(\frac{\ln(1/\delta) + (k-1) \ln(1/\beta)}{(\beta^k \eps_0)^2})$. The total number of epochs $k$ until \Cref{algo:GHAL} halts, by \Cref{thm:master} is bounded by $\lceil\max_{p \in\{0, 1, 2, 5/2\}}\log_{1/\beta}(\frac{A_p}{\eps})\rceil$ plus lower-order (log-log and constant) terms, thus $\frac{1}{\beta^{2k}} = \widetilde{\cO}\big(\frac{\max_{p \in\{0, 1, 2, 5/2\}}A_p^2}{\eps^2}\big),$ leading to the claimed sample complexity bound.
\end{proof}
\subsection{General, Bounded-Variance Setting}\label{sec:gen-bnd-var}

Let us consider first the most general setting, under which we only make the basic assumption stated in \eqref{eq:stochastic-opt-assmpts}. To bound the sample complexity of \Cref{algo:GHAL}, we primarily need to bound the sample sizes $m_{k, j}.$ The reason is that, as already remarked at the end of the previous subsection, $n_k$ is logarithmic in $1/\delta$ and polynomial in $k,$ thus it contributes polylogarithmic factors that get absorbed in the tilde-Oh notation, while the number of samples consumed by MoME calls was already bounded in \Cref{lemma:MoME-calls-samples}. %
The per-iteration number of samples $m_{k, j}$ needed to satisfy the condition in \Cref{thm:master} is bounded in the following lemma. To avoid overcomplicating expressions, since there is no apparent benefit to choosing $\beta$ as anything other than a universal constant within $(0, 1)$ (e.g., $\beta = 1/2$), in the rest of this section we state all the complexity results treating $1/\beta$ and $1/(1-\beta)$ as constants, so they are absorbed in the big-Oh notation.  

\begin{lemma}\label{lemma:mkj-general}
    Suppose that the stochastic oracle $\vtau(\vx; \vxi)$ satisfies \eqref{eq:stochastic-opt-assmpts} and consider any epoch $k \geq 1$ and iteration $j \geq 1$ in \Cref{algo:GHAL}. Let $\rho_k = \gammab(1-\lambda_k).$ Then taking
    \begin{equation}\label{eq:gen-mkj}
        m_{k, i} := \Big\lceil \frac{4 J_k^2}{\eps_k^2}\Big\rceil = \cO\Big(\frac{1}{\eps_k^2 \max\{(1-\gammab)^2, \lambda_k^2\}}\Big)%
    \end{equation}
    for all $k \geq 1, i \in [J_k],$ is sufficient to satisfy \Cref{asspt:variance-minibatch} used in \Cref{thm:master}, with \[a_0 = \sqrt{\kappa_\cE}\sigma(1  + 3\sqrt{2\kappa_\cE}\sqrt{\ln(4/((1-\beta)\delta))}),\, a_1 = 3\sqrt{2}\kappa_\cE\sigma\sqrt{\ln(1/\beta)}.\] 
\end{lemma}
\begin{proof}
    Fix $k \geq 1, i \in [J_k].$ %
    Recall that by \Cref{lemma:minibatching-variance} and the definition of $\sigma_{k, i},$ we have
\begin{equation}\notag%
    \sigma_{k, i}^2 \leq \frac{\kappa_\cE}{m_{k, i}}\Big(\frac{1}{m_{k, i}}\sum_{\ell=1}^{m_{k, i}}\E\norm{\vtau(\vy_i, \vxi_\ell^{(1)}) - \vtau(\vy_{i-1}, \vxi_\ell^{(2)}) - (\mT(\vy_i) - \mT(\vy_{i-1}))}^2\Big).
\end{equation} 
By Young's inequality and the bounded variance assumption in \eqref{eq:stochastic-opt-assmpts}, we have that \[\E\norm{\vtau(\vy_i, \vxi_\ell^{(1)}) - \vtau(\vy_{i-1}, \vxi_\ell^{(2)}) - (\mT(\vy_i) - \mT(\vy_{i-1}))}^2 \leq 4 \sigma^2;\] thus:
\begin{equation}\label{eq:bnd-on-sigma2}
    \sigma_{k, i}^2 \leq \frac{4\kappa_\cE\sigma^2 }{m_{k, i}}. 
\end{equation}
Consider $m_{k, i} = \lceil \frac{4 J_k^2}{\eps_k^2}\rceil.$ 
From \eqref{eq:bnd-on-sigma2}, under the stated choice of $m_{k, i}$, we have $\sigma_{k, i}^2 \leq \kappa_\cE\sigma^2 \eps_k^2/J_k^2.$ Thus:
\begin{align}\label{eq:sum-variance-bnd-gen}
    \sum_{i=1}^{J_k} \sigma_{k, i} &\leq %
    \eps_k \sqrt{\kappa_\cE} \sigma.
\end{align}
Notice similarly that
\(
    \sum_{i=1}^j \sigma_{k, i}^2 \leq %
    \eps_k^2 \kappa_\cE\sigma^2.  
\) 
Thus, recalling that $\delta_k = (1-\beta)\delta \beta^{k-1},$ we have
\begin{align}
    3\sqrt{2\kappa_\cE\ln(4/\delta_k)\sum_{i=1}^j \sigma_{k, i}^2} &\leq 3\sqrt{2}\kappa_\cE\sigma \eps_k\sqrt{\ln(4/((1-\beta)\delta)) + (k-1)\ln(1/\beta)}\notag\\
    &\leq 3\sqrt{2}\kappa_\cE\sigma \eps_k\big(\sqrt{\ln(4/((1-\beta)\delta))} + \sqrt{(k-1)\ln(1/\beta)}\big).\label{eq:2nd-var-term-gen}
\end{align}
To complete the proof, it remains to combine \eqref{eq:sum-variance-bnd-gen} and \eqref{eq:2nd-var-term-gen}. 
\end{proof}

Using \Cref{thm:master} and \Cref{lemma:mkj-general}, we can now bound the total sample complexity of \Cref{algo:GHAL} as follows.

\begin{corollary}\label{cor:sample-complexity-gen}
    Let $\eps > 0$ and $\delta \in (0, 1)$ be given parameters and suppose that $\vtau(\vx; \vxi)$ satisfies \eqref{eq:stochastic-opt-assmpts}. Define $D:= \norm{\vx_* - \vx_0} + \sigma.$ Then, the total number of samples utilized by \Cref{algo:GHAL} invoked with $\gammab = \gamma$, $\eps_0 = \lambdab_0 = 1$, and $\beta = 1/2$ is
    \begin{equation}\label{eq:gen-sample-complexity-total}
        \widetilde\cO_{\kappa_\cE, \ln(1/\delta)}\Big(\frac{D^2}{\eps^2} + \min\Big\{\frac{D^6}{\eps^5}, \, \frac{D^3}{(1-\gamma)^3 \eps^2}\Big\}\Big).
    \end{equation}
\end{corollary}
\begin{proof}
    Since $\|\mT(\vx_0) - \vx_0\| \leq 2 \|\vx_0 - \vx_*\|,$ under the stated parameter setting, we have that $\max_{p \in \{0, 1, 2, 5/2\}}A_p = \cO_{\kappa_\cE, \ln(1/\delta)}(\|\vx_0 - \vx_*\| + \sigma).$ Observe first that, using \Cref{lemma:MoME-calls-samples} and \Cref{thm:master}, the total number of samples utilized by calls to MoME is $\widetilde{\cO}_{\kappa_\cE, \ln(1/\delta)}\Big(\frac{D^2}{\eps^2}\Big)$, thus it is absorbed by the bound stated in \eqref{eq:gen-sample-complexity-total}.  
    
    Observe further that $J_k = \cO(\max\big\{\frac{1}{\lambda_k},\, \frac{1}{1-\gamma}\big\}\big).$ Thus, using \Cref{lemma:mkj-general}, 
    \begin{equation}\label{eq:sum-mkj-gen}
        \sum_{j=1}^{J_k}m_{k, j} = \cO\Big(\frac{1}{\eps_k^2}\min\Big\{\frac{1}{\lambda_k^3},\, \frac{1}{(1-\gamma)^3}\Big\}\Big) = \cO\Big\{\min\Big\{\frac{1}{\beta^{5k}},\, \frac{1}{\beta^{2k}(1-\gamma)^3}\Big\}\Big\}.
    \end{equation}
    Further, using the results of \Cref{thm:master}, we have $n_k = \cO(\mathrm{poly}(k)\max_{p \in \{0, 1, 2, 5/2\}}A_p) = \widetilde\cO_{\kappa_\cE, \ln(1/\delta)}((\|\vx_* - \vx_0\| + \sigma) \ln(1/\epsilon))$. Multiplying \eqref{eq:sum-mkj-gen} by $n_k,$ we get 
    \begin{equation}\label{eq:per-epoch-total-gen}
        n_k \sum_{j=1}^{J_k}m_{k, j} = \min\Big\{\frac{1}{\beta^{5k}},\, \frac{1}{\beta^{2k}(1-\gamma)^3}\Big\} \widetilde\cO_{\kappa_\cE, \ln(1/\delta)}((\|\vx_* - \vx_0\| + \sigma) \ln(1/\epsilon)),
    \end{equation}
    so now the only terms dependent on $k$ are of the form $1/\beta^{\alpha k}$ for $\alpha \in \{2, 5\}.$ Since the total number of epochs $K$ until the algorithm terminates is equal to $\lceil\log_{1/\beta}(8 \max_{p \in \{0, 1, 2, 5/2\}}A_p/\epsilon)\rceil$ plus lower-order (log-log and constant) terms, we have that $\sum_{k=1}^K \frac{1}{\beta^{\alpha k}} = \widetilde\cO\big(\big(\frac{\max_{p \in \{0, 1, 2, 5/2\}}A_p}{\epsilon}\big)^{\alpha}\big).$  The claimed bound on the sample complexity now follows by summing \eqref{eq:per-epoch-total-gen} over $k$.  
\end{proof}

It is worth pointing out here that we do not expect the true complexity to have a higher polynomial scaling with $\|\vx_* - \vx_0\| + \sigma$ than it does with $1/\eps.$ The technical reason for this comes from the choice of $n_k \propto \Rt_k/\eps_k.$ It seems plausible that taking $n_k \propto \Rt_k$ might lead to the same polynomial scaling in both $\|\vx_* - \vx_0\| + \sigma$ and $1/\eps;$ however, such a choice would result in a much more complicated recursion on $\Rt_k$ than \eqref{eq:Rtk-recursion-1}, involving all prior $\Rt_i,$ $i \leq k-1$ instead of just $\Rt_{k-1}.$ While we see this as a possible path towards tightening the stated complexity, we do not pursue it in this work.   

\subsection{Convergence of VR-GHAL under Expected Lipschitzness}\label{sec:exp-Lip}

Consider now the setting where, in addition to \eqref{eq:stochastic-opt-assmpts}, the stochastic oracle $\vtau$ also satisfies the multi-point oracle and expected Lipschitzness assumptions stated in \Cref{asspt:multi-query-oracle} and \Cref{asspt:Lipschitz-in-expectation}. We argue that in this case, the sample complexity can be improved, as stated in \Cref{cor:exp-Lip}.

As before, to apply \Cref{thm:master}, we first need to bound the sample sizes $m_{k, j}$, which is done in the following lemma.

\begin{lemma}\label{lemma:samples-exp-Lip}
    Suppose that the stochastic oracle $\vtau(\vx; \vxi)$ satisfies \eqref{eq:stochastic-opt-assmpts}, \Cref{asspt:multi-query-oracle}, and \Cref{asspt:Lipschitz-in-expectation}. Consider epochs $k \geq 1$ and iterations $j \geq 1$ in \Cref{algo:GHAL}. Let $\rho_k = \gammab(1-\lambda_k).$ Then taking
    \begin{equation}\label{eq:Lip-exp--mkj}
        m_{k, j} := \Big\lceil \frac{4\Rt_{k-1}^2}{\eps_k^2 (1-\rho_k)^2}\Big\rceil = \cO\Big(\frac{\max_{p \in\{0, 1, 2, 5/2\}}(A_p/\eps_0)^2 \mathrm{poly}(k)}{\max\{(1-\gammab)^2, \lambda_k^2\}}\Big)
    \end{equation}
    for all $k \geq 1, j \in [J_k],$ is sufficient to satisfy \Cref{asspt:variance-minibatch} with \[a_0 = \sqrt{L^2 + \gamma^2}\sqrt{\kappa_\cE}(1 + 3\sqrt{2\kappa_\cE}\sqrt{\ln(4/((1-\beta)\delta))}),\, a_1 = 3\sqrt{\kappa_\cE}\sqrt{2(L^2 + \gamma^2)\ln(1/\beta)}.\] 
\end{lemma}
\begin{proof}
    Recall that under Assumptions \ref{asspt:multi-query-oracle} and \ref{asspt:Lipschitz-in-expectation}, $\vxi_\ell^{(1)} = \vxi_\ell^{(2)}$ in the definition of the difference operator (see \eqref{eq:delta-kj-def}). As a consequence, we have that:
    \begin{align}
        \sigma_{k, i}^2 &\leq \frac{\kappa_\cE}{m_{k, i}}\Big(\frac{1}{m_{k, i}}\sum_{\ell=1}^{m_{k, i}}\E\norm{\vtau(\vy_i, \vxi_\ell^{(1)}) - \vtau(\vy_{i-1}, \vxi_\ell^{(1)}) - (\mT(\vy_i) - \mT(\vy_{i-1}))}^2\Big)\notag\\
        &\leq \frac{2\kappa_\cE}{m_{k, i}}\Big(\frac{1}{m_{k, i}}\sum_{\ell=1}^{m_{k, i}}\big(\E\big[\|\vtau(\vy_i, \vxi_\ell^{(1)}) - \vtau(\vy_{i-1}, \vxi_\ell^{(1)})\|^2\big] + \|\mT(\vy_i) - \mT(\vy_{i-1}))\|^2\big)\Big)\notag\\
        &\leq \frac{4\kappa_\cE(L^2 + \gamma^2) \norm{\vy_i - \vy_{i-1}}^2}{m_{k, i}}.\label{eq:var-Lip-exp-1}
    \end{align}
    where the second inequality is by Young's inequality, and the third inequality follows from $\gamma$-Lipschitzness of $\mT$ and expected Lipschitzness of $\vtau$ (\Cref{asspt:Lipschitz-in-expectation}). 

    Recall (from proof of \Cref{lemma:basic-GHAL-inequalities}) that $\norm{\vy_i - \vy_{i-1}}^2 \leq \rho_k^{2(i-1)}(1-\lambda_k)^2\Rt_{k-1}^2.$ Let $m_{k,i}$ be defined as in the lemma statement. Then, \eqref{eq:var-Lip-exp-1} implies
    \begin{equation}\notag
        \sigma_{k, i}^2 \leq (L^2 + \gamma^2)\kappa_\cE\eps_k^2\rho_k^{2(i-1)}(1-\rho_k)^2.
    \end{equation}
    Thus, $\sum_{i=1}^{J_k}\sigma_{k, i} \leq \eps_k\sqrt{\kappa_\cE(L^2 + \gamma^2)}$ and $\sum_{i=1}^{J_k}\sigma_{k, i}^2 \leq \eps_k^2\kappa_\cE (L^2 + \gamma^2).$ The rest of the calculation is similar to the proof of \Cref{lemma:mkj-general}, with $\sigma^2$ replaced by $L^2 + \gamma^2,$ and is thus omitted.  
\end{proof}

Using \Cref{lemma:samples-exp-Lip} to apply \Cref{thm:master}, we now get the following sample complexities.

\begin{corollary}\label{cor:exp-Lip}
     Let $\eps > 0$ and $\delta \in (0, 1)$ be given parameters and suppose that $\vtau(\vx; \vxi)$ satisfies \eqref{eq:stochastic-opt-assmpts}, \Cref{asspt:multi-query-oracle}, and \Cref{asspt:Lipschitz-in-expectation}. Let $D_L := \norm{\vx_* - \vx_0}+ \sigma + \sqrt{L^2 + \gamma^2}.$ Then, the total number of samples utilized by \Cref{algo:GHAL} invoked with $\gammab = \gamma$, $\eps_0 = \lambdab_0 = 1$, and $\beta = 1/2$ is
    \begin{equation}\label{eq:exp-Lip-sample-complexity-total}
    \begin{aligned}
        \widetilde\cO_{\kappa_\cE, \ln(1/\delta)}\Big(&\frac{D_L^2}{\eps^2} + \min\Big\{\frac{D_L^6}{\eps^3}, \, \frac{D_L^3}{(1-\gamma)^3}\Big\}\Big).
    \end{aligned}
    \end{equation}
\end{corollary}
\begin{proof}
Similar to \Cref{cor:sample-complexity-gen}, under the stated parameter setting, we have that $\max_{p \in \{0, 1, 2, 5/2\}}A_p = \cO_{\kappa_\cE, \ln(1/\delta)}(\|\vx_0 - \vx_*\| + \sigma + \sqrt{L^2 + \gamma^2}).$ Using \Cref{lemma:MoME-calls-samples} and \Cref{thm:master}, the total number of samples utilized by calls to MoME is $\widetilde{\cO}_{\kappa_\cE, \ln(1/\delta)}\Big(\frac{D_L^2}{\eps^2}\Big)$, thus it is absorbed by the bound stated in \eqref{eq:exp-Lip-sample-complexity-total}.     

Thus, in the rest of the proof, we focus on bounding the sample complexity of difference-based estimators, amounting to $\sum_{k=1}^K n_k\sum_{j=1}^{J_k}m_{k, j},$ where $K$ is the total number of epochs, which we have already bounded in \Cref{thm:master}. Recall that $K$ is equal to $\log_{1/\beta}(8\max_{p \in \{0, 1, 2, 5/2\}}A_p/\eps)$ plus lower-order (constant and log-log) terms. Thus, any polynomial dependence on $k \in [K]$ gets absorbed by the tilde-Oh notation in \eqref{eq:exp-Lip-sample-complexity-total}. 

By \Cref{thm:master}, we can conclude that $n_k = \cO(\mathrm{poly}(k)\max_{p \in \{0, 1, 2, 5/2\}}A_p) = \widetilde\cO_{\kappa_\cE, \ln(1/\delta)}((\|\vx_* - \vx_0\| + \sigma + \sqrt{L^2 + \gamma^2}) \ln(1/\epsilon))$. Additionally, since for all $k \geq 1,$ $J_k = \cO(\frac{1}{\max\{1-\gammab, \lambda_k\}}),$ we have, using \Cref{lemma:samples-exp-Lip}, that $\sum_{j=1}^{J_k}m_{k, j} = \cO\Big(\frac{(\|\vx_* - \vx_0\| + \sigma + \sqrt{L^2 + \gamma^2})^2 \mathrm{poly}(K)}{\max\{(1-\gammab)^3, \lambda_k^3\}}\Big)$. Observe that for $1 - \gammab = \Omega(\lambda_k),$ $\sum_{k=1}^K n_k\sum_{j=1}^{J_k}m_{k, j}$ is independent of $k$; otherwise it scales with $1/\beta^{3k}$ due to the $\lambda_k^3$-dependence. As before, it remains to use that $\sum_{k=1}^K \frac{1}{\beta^{3k}} = \cO\big(\big(\frac{1}{\beta}\big)^{3K}\big),$ leading to the claimed bound.  
\end{proof}
\subsection{Convergence under Other Variance Models}\label{sec:other-models}

Finally, we briefly comment on other possible variance models that could be plausibly addressed using the techniques from this work. 

First, if each sample is $\gamma$-Lipschitz, then the clipping operation in the definition of $\mTt_{k, j}$ is not needed (or, if employed, the clipping would never occur). As a result, the bound in \Cref{prop:variance-of-estimator} can be strengthened so that it does not include $\sum_{i=1}^j \sigma_{k, i}$, which came from bounding the clipping bias terms $\E[\cldel_{k, i}^{(\ell)}|\cF_{k, 0}^{(0)}] - \E[\Delta_{k, i}^{(\ell)}|\cF_{k, 0}^{(0)}]$ that in this case are zero. It is not hard to verify that in this case it suffices to take $m_{k, j} := \Big\lceil \frac{4\Rt_{k-1}^2}{\eps_k^2 (1-\rho_k)}\Big\rceil = \cO\Big(\frac{\max_{p \in\{0, 1, 2, 5/2\}}(A_p/\eps_0)^2 \mathrm{poly}(k)}{\max\{(1-\gammab), \lambda_k\}}\Big)$ to satisfy the requirement from \Cref{thm:master}, since $\sum_{i=1}^{\infty}\rho_k^{2i} = \cO(\frac{1}{1 - \rho_k}).$ The result is that the sample complexity corresponds to an improved version of what was obtained in \Cref{cor:exp-Lip}, with $\eps^3$ replaced by $\eps^2$, without needing \Cref{asspt:Lipschitz-in-expectation} and with constants dependent only on the samplewise Lipschitz parameter. In other words, the sample complexity scaling with $1/\eps$ improves to $1/\eps^2$ for nonexpansive operators (for contractive operators, \Cref{cor:exp-Lip} already gives such a scaling). This result is summarized in the following corollary, the proof of which is omitted for brevity (it is immediate based on the aforementioned observations).

\begin{corollary}\label{cor:samplewise-Lipschitz}
     Let $\eps > 0$ and $\delta \in (0, 1)$ be given parameters and suppose that $\vtau(\vx; \vxi)$ satisfies \eqref{eq:stochastic-opt-assmpts}, \Cref{asspt:multi-query-oracle}, and \Cref{asspt:samplewise-Lip}. Let $D := \norm{\vx_* - \vx_0}+ \sigma.$ Then, the total number of samples utilized by \Cref{algo:GHAL} invoked with $\gammab = \gamma$, $\eps_0 = \lambdab_0 = 1$, and $\beta = 1/2$ is
    \begin{equation}\label{eq:samplewise-Lip-sample-complexity-total}
    \begin{aligned}
        \widetilde\cO_{\kappa_\cE, \ln(1/\delta)}\Big(&\frac{D^2}{\eps^2} + \min\Big\{\frac{D^4}{\eps}, \, \frac{D^3}{(1-\gamma)}\Big\}\Big).
    \end{aligned}
    \end{equation}
\end{corollary}

Second, different ``unbounded variance'' models exist in the literature. For instance, models replacing the second-moment bound in \eqref{eq:stochastic-opt-assmpts} by
\begin{equation}\label{eq:unbounded-var-alpha}
    \E_{\vxi \sim \cD}[\|\vtau(\vx; \vxi) - \mT(\vx)\|^\alpha] \leq \sigma^\alpha
\end{equation}
for some $\alpha \in (1, 2)$, 
or
\begin{equation}\label{eq:unbounded-var-dist}
    \E_{\vxi \sim \cD}[\|\vtau(\vx; \vxi) - \mT(\vx)\|^2] \leq \sigma_0^2 + \sigma_1^2 \norm{\mT(\vx) - \mT(\vx_*)}^2.
\end{equation}
It seems plausible that our techniques can be adapted to both models, for the following reasons. First, for the model in \eqref{eq:unbounded-var-alpha}, there are recent results for mean estimation in Banach spaces that can be utilized in place of \Cref{prop:Banach-concentration}; see, for instance, \cite{whitehouse2024mean}. We expect the sample complexity to deteriorate in this case as $\alpha$ is decreased, due to the higher sample complexity of mean estimation appearing in such settings. Second, the model in \eqref{eq:unbounded-var-dist} appears similarly possible to address because our analysis already deals with bounding the distance of the iterates to the fixed point (see \Cref{lemma:basic-GHAL-inequalities}), which already had to be accounted for in the error budget in \Cref{thm:master} when bounding $\Rt_k.$  Thus, a more careful analysis additionally keeping track of terms resulting from the $\sigma_1^2 \norm{\mT(\vx) - \mT(\vx_*)}^2$ terms appears reachable. We leave such considerations to future work. 

\section{Conclusion}

We introduced a high-probability framework for solving stochastic fixed-point equations in the native norm of a quadratically smoothable Banach space. The proposed VR-GHAL method combines the gradual Halpern framework from \cite{diakonikolas2025pushing} with a recursive variance-reduced estimator built from Lipschitz-clipped stochastic differences. This clipping rule is tailored to fixed-point problems: by clipping at the scale dictated by the operator Lipschitz bound, the estimator preserves the pathwise stability needed by the gradual Halpern analysis, while Banach-space martingale concentration yields (poly-)logarithmic dependence on the failure probability under only finite second moments.

The resulting guarantees address three limitations of prior stochastic fixed-point methods simultaneously. They are high-probability rather than merely expectation-based, they measure both variance and contraction in the native norm rather than through Euclidean reductions, and they are anytime/parameter-free with respect to problem quantities such as the variance proxy and the distance to a fixed point. In the general bounded-variance setting, the method matches the best known polynomial dependence on the target accuracy up to logarithmic factors while strengthening the guarantee to high probability and extending the geometry beyond finite-dimensional Euclidean variance assumptions. Under additional oracle regularity, the rates improve, recovering $\eps^{-2}$-type behavior in the samplewise nonexpansive regime.

Several questions remain open. First, the dependence on the initial distance and variance proxy in the most general bounded-variance bound is likely not tight; improving the recursion for the epoch-start estimator may reduce the current higher-order dependence on $D = \norm{\vx_* - \vx_0} + \sigma$. Second, it would be valuable to establish matching lower bounds in native Banach-space geometries, especially for nonexpansive operators under finite second moments. Third, the same clipped-difference approach may extend to heavier-tailed oracle models with only finite $\alpha$-moments, state-dependent variance, or Markovian sampling. Finally, developing fully explicit implementations for reinforcement learning and scientific-computing fixed-point problems would help clarify when the native-norm high-probability theory leads to practical improvements over Euclidean variance-reduction methods.

\bibliographystyle{alpha}
\bibliography{references,addl-refs}

\appendix

\section{Omitted Proofs from \Cref{sec:prelims}}\label{appx:omitted-prelims}

In this section, we restate and prove auxiliary claims from \Cref{sec:prelims}.

\lemmaminibatchvar*
\begin{proof}
    Define $\mS_0 := 0,$ $\mS_i := \sum_{j=1}^i \veta_j,$ for $i \in [m].$ By \Cref{asspt:Banach}, there exists a compatible norm $\norm{\cdot}_+$ and constants $B, C \geq 1$ such that
    \begin{equation}\label{eq:sum-step}
        \norm{\mS_i}_+^2 \leq \norm{\mS_{i-1}}_+^2 + D(\norm{\cdot}_+^2)(\mS_{i-1})[\veta_i] + B\norm{\veta_i}_+^2.
    \end{equation}
    By definition, $\mS_{i-1}$ is $\cF_{i-1}$-measurable, $D(\norm{\cdot}_+^2)$ is continuous, and $\E[\veta_i|\cF_{i-1}] = \vzero$. Because $\|D(\norm{\cdot}_+^2)(\vx)\|_{+,*} = 2\norm{\vx}_+$, it is integrable. Thus, $\E[D(\norm{\cdot}_+^2)(\mS_{i-1})[\veta_i]|\cF_{i-1}] = 0.$ Taking conditional expectations w.r.t.\ $\cF_{i-1}$ on both sides of \eqref{eq:sum-step}, we thus have
    \begin{equation}\label{eq:iterated-sum-+}
        \E[\norm{\mS_i}_+^2|\cF_{i-1}] \leq \E[\norm{\mS_{i-1}}_+^2|\cF_{i-1}] + B\E[\norm{\veta_i}_+^2|\cF_{i-1}].
    \end{equation}
    By the tower property of expectation, 
    \[
        \E[\norm{\mS_i}_+^2|\cF_{0}] \leq \E[\norm{\mS_{i-1}}_+^2|\cF_{0}] + B\E[\norm{\veta_i}_+^2|\cF_{0}].
    \]
    Unrolling the recursion in the last inequality and using that $\norm{\cdot} \leq \norm{\cdot}_+ \leq C \norm{\cdot}$, we get
    \begin{align}
        \E[\norm{\mS_i}^2|\cF_{0}]\leq \E[\norm{\mS_i}_+^2|\cF_{0}] &\leq B\sum_{j=1}^i \E[\norm{\veta_j}_+^2|\cF_{0}]\notag\\
        &\leq BC^2\sum_{j=1}^i \E[\norm{\veta_j}^2|\cF_{0}]\notag\\
        &= \kappa_{\cE}\sum_{j=1}^i \E[\norm{\veta_j}^2|\cF_{0}]. \notag%
    \end{align}
    The first lemma statement now follows by taking $i = m$ and dividing both sides by $m^2.$  
    
    The remaining statement follows by plugging $\E[\norm{\veta_i}_+^2|\cF_{i-1}] \leq C^2 \E[\norm{\veta_i}^2|\cF_{i-1}] \leq C^2 \sigma_i^2$ into \eqref{eq:iterated-sum-+} and repeating the same argument.
\end{proof}

\propbanachconc*
\begin{proof}
    Because $(\cE, \norm{\cdot})$ is quadratically smoothable, there exists a compatible norm $\norm{\cdot}_+$ and the associated constants $B, C \geq 1,$ $\kappa_\cE:= BC^2$, such that \Cref{asspt:Banach} ensures that the space $(\cE, \norm{\cdot}_+)$ is $(2, \kappa_\cE)$-smooth. Fix such a norm $\norm{\cdot}_+$ and the associated constants. Assume that $b_i > 0;$ otherwise the corresponding summand $\veta_i$ is a zero element of $\cE$ almost surely, and can thus be omitted from the sum $\mS_m$ (if all $b_i = 0$, the claimed bound holds trivially). Throughout the proof, we condition on $\cF_0,$ so the parameters $b_i, \sigma_i^2$ are fixed. This conditioning is w.l.o.g., as it suffices to use the tower property of expectation at the end to reach the claimed unconditional statement.  
    
     The main idea in the proof, following \cite{pinelis1994optimum},\footnote{Here, we make use of the observation that $(\cE, \norm{\cdot}_+)$ defines a separable smooth Banach space, so the arguments of \cite{pinelis1994optimum} apply. Then, it suffices to use compatibility of $\norm{\cdot}, \norm{\cdot}_+$ from \Cref{asspt:Banach} to translate the results to $\norm{\cdot}.$ A similar idea is also used in \cite{juditsky2008large}.} is to fix an arbitrary $\rho > 0$ and construct a non-negative supermartingale $M_0, M_1, \dots, M_n$ defined by $M_0 = 1$ and $M_m := \cosh(\rho\norm{ S_m}_+) \exp\big(-A_m\big)$ for a suitable choice of a sequence $A_m$. An application of Ville's inequality, stated below, then leads to a time-uniform bound on $M_m, 0\leq m \leq n$, which can then be translated into the claimed bound from the proposition statement.
\begin{fact}[Ville's Inequality \cite{ville1939etude}]\label{fact:Ville}
    Let $M_0, M_1, \dots, M_n$ be a non-negative supermartingale. Then, for any $t > 0,$ 
    \begin{equation}\notag
        \prob[\sup_{0\leq m \leq n}M_m \geq t] \leq \frac{\E[M_0]}{t}. 
    \end{equation}
\end{fact}
    Following similar arguments as in \cite{pinelis1994optimum,whitehouse2024mean}, we first bound $\E[\cosh(\rho\norm{ \mS_m}_+)|\cF_{m-1}]$ by bounding $\phi(\theta) := \E[\cosh(\rho \norm{S_{m-1} + \theta \veta_m}_+)| \cF_{m-1}]$, $\phi: [0, 1] \to [1, \infty)$, noting that $\phi(1) = \E[\cosh(\rho\norm{ \mS_m}_+)|\cF_{m-1}]$. As argued in \cite{pinelis1994optimum}, although the norm $\norm{\cdot}_+$ (and thus also $\phi$) is not  necessarily twice differentiable, it can effectively be treated as everywhere Fr\'echet-differentiable with Fr\'echet derivatives of arbitrary order, using an appropriate Gaussian smoothing with diminishing smoothing radius; see \cite[Lemma 2.2 and Remark 2.4]{pinelis1994optimum}. In particular, the argument used in proving \cite[Theorem 3.2]{pinelis1994optimum} (see also the proof of \cite[Proposition 3.3]{whitehouse2024mean}) implies that $\phi'(0) = 0$ and %
    \begin{align*}
        \phi''(\theta) &\leq \rho^2B \E[\norm{\veta_m}_+^2 \cosh(\rho \norm{\mS_{m-1} + \theta \veta_m}_+)|\cF_{m-1}]\\
        & \leq \rho^2 B \E[\norm{\veta_m}_+^2 \cosh(\rho\norm{\mS_{m-1}}_+)e^{\rho \theta \norm{\veta_m}_+}|\cF_{m-1}]\\
        &\leq \rho^2 B \cosh(\rho\norm{\mS_{m-1}}_+) e^{C b_m \rho \theta}  C^2 \sigma_m^2,
    \end{align*}
    where the last inequality follows from $\norm{\veta_m}_+ \leq C \norm{\veta_m} \leq C b_m$ and $\E[\norm{\veta_m}_+^2|\cF_{m-1}] \leq C^2\E[\norm{\veta_m}^2|\cF_{m-1}] \leq C^2\sigma_m^2$, which both hold by assumption. As a result, applying the second-order Taylor expansion to $\phi$, we have
    \begin{align}
        \E[\cosh(\rho\norm{ \mS_m}_+)|\cF_{m-1}] &= \phi(1) = \phi(0) + \phi'(0) + \int_0^1(1-\theta)\phi''(\theta)\dd \theta \notag\\
        &\leq \cosh(\rho\norm{\mS_{m-1}}_+)\big(1 + \rho^2 B C^2\sigma_m^2 \int_0^1 (1- \theta) e^{C b_m \rho \theta} \dd \theta\big) \notag\\
        &\leq \cosh(\rho\norm{\mS_{m-1}}_+)\Big(1 + \rho^2 \kappa_{\cE} \sigma_m^2  \frac{C e^{b_m \rho} - C b_m \rho - 1}{(C b_m \rho)^2}\Big) \notag\\
        &\leq \cosh(\rho\norm{\mS_{m-1}}_+)\exp\Big( \kappa_\cE (\sigma_m/b_m')^2  \big({e^{b_m' \rho} - b_m' \rho - 1}\big)\Big), \label{eq:nonneg-supm}
    \end{align}
    where we used that $\kappa_\cE = BC^2,$ $\int_0^1(1-\theta)e^{a \theta} \dd \theta = \frac{e^a - a - 1}{a^2}$, and $e^x \geq 1 + x$, $\forall x \in \sR$, and we denoted for simplicity $b_m' := C b_m$. 
    
    Define $A_0 := 0$, $A_m := \kappa_\cE \sum_{i=1}^m (\sigma_i/b_i')^2  ({e^{b_i' \rho} - b_i' \rho - 1})$. A rearrangement of \eqref{eq:nonneg-supm} now implies that \[M_m := \cosh(\rho\norm{ \mS_m}_+) \exp\big(-A_m\big)\] is a non-negative supermartingale with initial value $M_0 = 1$. Thus, by \Cref{fact:Ville}, $\prob[\sup_{m \geq 0}M_m \geq t] \leq \frac{1}{t},$ for any $t > 0.$ Since $\cosh(\rho\norm{ \mS_m}_+) \geq \frac{1}{2}e^{\rho\norm{ \mS_m}_+}$ and $\norm{\mS_m} \leq \norm{\mS_m}_+,$ we get, as a consequence, that for any $t > 0$ and any fixed $n \geq 1,$
    \begin{equation}
        \prob\Big[\sup_{1 \leq m \leq n} e^{\rho\norm{ \mS_m} -A_m } \geq 2 t\Big] \leq \frac{1}{t}.
    \end{equation}
    In particular, setting $t = 1/\delta$, we have that  with probability $1-\delta,$ for all $m \in  [ n],$ 
    \begin{align*}
        \norm{\mS_m} &\leq \frac{A_m + \ln(2/\delta)}{\rho}\\
        &\leq \frac{\kappa_\cE\sum_{i=1}^m(\sigma_i/b_i')^2(e^{b_i' \rho} - b_i' \rho - 1) + \ln(2/\delta)}{\rho}.
    \end{align*}
    To simplify the last inequality, we use the standard derivation common to Bernstein-style inequalities; see, for example, the proof of Proposition 2.1.4 in \cite{wainwright2019high}. Note first that $h(t) = \frac{e^t - t - 1}{t^2}$ is increasing in $t$. As a consequence, $\frac{e^{b_i' \rho} - b_i' \rho - 1}{{(b_i')}^2} \leq \frac{e^{\norm{\vb_{1:n}'}_{\infty} \rho} - \norm{\vb_{1:n}'}_{\infty} \rho - 1}{\norm{\vb_{1:n}'}_{\infty}^2}$, where $\norm{\vb_{1:n}'}_{\infty} = C \norm{\vb_{1:n}}_{\infty},$ and the above inequality simplifies to
    \begin{align*}
        \norm{\mS_m} 
        &\leq \frac{\kappa_\cE\sum_{i=1}^m(\sigma_i/\norm{\vb_{1:n}'}_{\infty})^2(e^{\norm{\vb_{1:n}'}_{\infty} \rho} - \norm{\vb_{1:n}'}_{\infty} \rho - 1) + \ln(2/\delta)}{\rho}.
    \end{align*}
   Further, for $\rho \norm{\vb_{1:n}'}_{\infty} = t <3$, we have that $h(t) \leq \frac{1}{2}(1- t/3)^{-1} = \frac{1}{2(1-\rho \norm{\vb_{1:n}'}_{\infty}/3)},$ and so we have
   \begin{align*}
       \norm{\mS_m} 
        &\leq \frac{\kappa_\cE\sum_{i=1}^m\frac{{\sigma_i}^2 \rho^2}{2(1-\rho \norm{\vb_{1:n}'}_{\infty}/3)} + \ln(2/\delta)}{\rho}\\
        &\leq \frac{\kappa_\cE\sum_{i=1}^n\frac{{\sigma_i}^2 \rho^2}{2(1-\rho \norm{\vb_{1:n}'}_{\infty}/3)} + \ln(2/\delta)}{\rho}\\
        &= \frac{\kappa_\cE \norm{\vsigma}_2^2 \rho}{2(1-\rho \norm{\vb_{1:n}'}_{\infty}/3)} + \frac{\ln(2/\delta)}{\rho},
   \end{align*}
   where, to simplify the notation, we denoted $\norm{\vsigma}_2^2 := \sum_{i=1}^n \sigma_i^2.$ 
   It remains to choose $\rho < 3/\norm{\vb_{1:n}'}_{\infty}$ and simplify the above expression. If $\norm{\vsigma}_2^2 = 0,$ then the claimed bound on $\norm{\mS_m}$ holds trivially, so assume this is not the case. Choosing $\rho = \frac{\sqrt{2\ln(2/\delta)/(\kappa_\cE \norm{\vsigma}_2^2)}}{1 + (\norm{\vb_{1:n}'}_{\infty}/3)\sqrt{2\ln(2/\delta)/(\kappa_\cE \norm{\vsigma}_2^2)}} < 3/\norm{\vb_{1:n}'}_{\infty},$ plugging into the above inequality, and simplifying, we get
   \begin{align*}
       \norm{\mS_m} 
        &\leq \sqrt{2\kappa_\cE \norm{\vsigma}_2^2 \ln(2/\delta)} + \frac{\ln(2/\delta) \norm{\vb_{1:n}'}_{\infty}}{3},
   \end{align*}
   completing the proof.
\end{proof}

\section{Omitted Proofs from \Cref{sec:VR}}\label{appx:omitted-clipping}

This section restates and proves \Cref{lemma:clipped-op-properties}. 

\lemmaclipping*
\begin{proof}
    Part 1 is immediate by the triangle inequality and, since, by \eqref{eq:clipped-diff-op}, $\|\cldel(\vx, \vy)\| \leq \gammab\|\vx - \vy\|$.

    For Parts 2 and 3, define $\chi = \one\{\|\Delta(\vx, \vy)\| > \gammab\|\vx - \vy\|\}$---this is the indicator of the event that the clipping operation has any effect on $\Delta(\vx, \vy)$, in which case $\Delta(\vx, \vy)$ gets truncated to a vector of length $\gammab\|\vx - \vy\|$ while not changing its direction. Observe that %
    \begin{align}
        \cldel(\vx, \vy) &= \Delta(\vx, \vy)(1 - \chi) + \frac{\gammab \|\vx - \vy\|}{\|\Delta(\vx, \vy)\|}\Delta(\vx, \vy) \chi\notag\\
        &= \Delta(\vx, \vy) + \Big(\frac{\gammab \|\vx - \vy\|}{\|\Delta(\vx, \vy)\|} - 1 \Big)\Delta(\vx, \vy) \chi. \notag
    \end{align}
    As a consequence,
    \begin{equation}\label{eq:clipped-minus-exp}
        \cldel(\vx, \vy) - \Delta(\vx, \vy) =   \Big(\frac{\gammab \|\vx - \vy\|}{\|\Delta(\vx, \vy)\|} - 1 \Big)\Delta(\vx, \vy) \chi.
    \end{equation}
    Observe that, by the definition of $\chi,$ we have
    \begin{align}
        \Big\|\Big(\frac{\gammab \|\vx - \vy\|}{\|\Delta(\vx, \vy)\|} - 1 \Big)\Delta(\vx, \vy) \chi\Big\| &= \Big(1 - \frac{\gammab \|\vx - \vy\|}{\|\Delta(\vx, \vy)\|} \Big)\|\Delta(\vx, \vy)\|\chi\notag\\
        &= \big(\|\Delta(\vx, \vy)\| - \gammab\|\vx - \vy\|\big)\chi \notag\\
        &\stackrel{(i)}{\leq} \big(\|\Delta(\vx, \vy)\| - \|\E[\Delta(\vx, \vy)]\|\big)\chi\notag\\
        &\stackrel{(ii)}{\leq} \|\Delta(\vx, \vy) - \E[\Delta(\vx, \vy)]\|\chi, \label{eq:clipped-portion-of-delta}
    \end{align}
    where $(i)$ holds because $\E[\Delta(\vx, \vy)]$ is $\gamma$-Lipschitz and $\gamma \leq \gammab,$ thus $\gammab\|\vx - \vy\| \geq \|\E[\Delta(\vx, \vy)]\|$, and $(ii)$ is by the triangle inequality (a rearrangement of $\|\Delta(\vx, \vy)\| \leq \|\E[\Delta(\vx, \vy)]\| + \|\Delta(\vx, \vy) - \E[\Delta(\vx, \vy)]\|$). %

    To prove Part 2 of the lemma, %
    notice that taking expectations on both sides of \eqref{eq:clipped-minus-exp}, we get $\E[\cldel(\vx, \vy)] - \E[\Delta(\vx, \vy)] = \E\big[ \big(\frac{\gammab \|\vx - \vy\|}{\|\Delta(\vx, \vy)\|} - 1 \big)\Delta(\vx, \vy) \chi\big].$ Thus, using \eqref{eq:clipped-portion-of-delta}, we have
    \begin{align}
        \|\E[\cldel(\vx, \vy)] - \E[\Delta(\vx, \vy)]\| &\leq  \Big\|\E\Big[\Big(\frac{\gammab \|\vx - \vy\|}{\|\Delta(\vx, \vy)\|} - 1 \Big)\Delta(\vx, \vy) \chi\Big]\Big\|\notag\\
        &\leq \E\Big[\Big\|\Big(\frac{\gammab \|\vx - \vy\|}{\|\Delta(\vx, \vy)\|} - 1 \Big)\Delta(\vx, \vy) \chi\Big\|\Big]\notag\\
        &\stackrel{(i)}{\leq} \E\big[\|\Delta(\vx, \vy) - \E[\Delta(\vx, \vy)]\|\chi\big]\notag\\
        &\stackrel{(ii)}{\leq} \sqrt{\E\big[\|\Delta(\vx, \vy) - \E[\Delta(\vx, \vy)]\|^2]\E[\chi^2]}\notag\\
       &\stackrel{(iii)}{\leq} \sigma_{\vx\vy}, %
       \label{eq:diff-of-expectations}
    \end{align}
    where $(i)$ is by \eqref{eq:clipped-portion-of-delta}, $(ii)$ is by Cauchy-Schwarz inequality, and $(iii)$ is by $\chi^2 \leq 1.$ %
    This completes the proof of Part 2.

    Further, using \eqref{eq:clipped-minus-exp}, squaring both sides of \eqref{eq:clipped-portion-of-delta} and taking expectations on both sides, we also have
    \begin{align}
        \E\big[\|\cldel(\vx, \vy) - \Delta(\vx, \vy)\|^2\big] &\leq \E\big[\|\Delta(\vx, \vy) - \E[\Delta(\vx, \vy)]\|^2\chi^2\big]\notag\\
        &\leq \E\big[\|\Delta(\vx, \vy) - \E[\Delta(\vx, \vy)]\|^2\big]\notag\\
        &\leq \sigma_{\vx \vy}^2. \label{eq:exp-cldel-del-squared}
    \end{align}

    For Part 3, %
    by Young's inequality,
    \begin{align*}
        \|\cldel(\vx, \vy) - \E[\cldel(\vx, \vy)]\|^2 \leq  3\big(&\|\cldel(\vx, \vy) - \Delta(\vx, \vy) \|^2 + \|\Delta(\vx, \vy) - \E[\Delta(\vx, \vy)]\|^2\\
        &+ \|\E[\Delta(\vx, \vy)]- \E[\cldel(\vx, \vy)]\|^2\big).
    \end{align*}
    Taking expectations on both sides of last inequality and using \eqref{eq:diff-of-expectations}, \eqref{eq:exp-cldel-del-squared}, and the lemma assumptions, we get that each summand is bounded by $\sigma_{\vx \vy}^2,$ completing the proof of Part 3. 
\end{proof}

\end{document}